\documentclass[10pt, reqno]{amsart}

\usepackage{chngcntr}
\usepackage{apptools}

\calclayout

\def\BibTeX{{\rm B\kern-.05em{\sc i\kern-.025em b}\kern-.08em
    T\kern-.1667em\lower.7ex\hbox{E}\kern-.125emX}}

\newtheorem{thm}{Theorem}[section]

\newtheorem{lem}[thm]{Lemma}
\newtheorem{prop}[thm]{Proposition}

\newtheorem{lemma}[thm]{Lemma}

\theoremstyle{definition}

\theoremstyle{remark}
\newtheorem{rem}{Remark}[section]

\numberwithin{equation}{section}

    \newcommand{\floor}[1]{\lfloor#1\rfloor}

    \newcommand{\EE}{\mathbb{E}}

    \newcommand{\Exp}{\operatorname{E}}
    \newcommand{\E}{\Exp}

    \renewcommand{\Pr}{\operatorname{P}}

    \newcommand{\dto}{\xrightarrow{d}}
    
    \newcommand{\vto}{\xrightarrow{v}}
     \newcommand{\Pto}{\xrightarrow{P}}
    
    \renewcommand{\d}{\mathrm{d}}

    \newcommand{\rmd}{\mathrm{d}}

\newcommand{\PP}{\mathrm{P}}

\newcommand{\be}{\begin{equation}}
    \newcommand{\ee}{\end{equation}}

\begin{document}

\title[Limit theory for a class of non-stationary linear processes] 
{Limit theory for a class of non-stationary linear processes with weakly dependent heavy-tailed innovations}

%



\author{Danijel Krizmani\'{c}}






\begin{abstract}
For a class of non-stationary linear processes $(X_{n})_{n}$ with weakly dependent heavy-tailed innovations we prove a convergence result for the space-time point processes $N_{n}=\sum_{i=1}^{n}\delta_{(i/n, X_{i}/a_{n})}$. Using this result and the continuous mapping theorem we then derive functional convergence of
the corresponding partial maxima process
in the space of c\`{a}dl\`{a}g functions on $[0,1]$.
\end{abstract}

\maketitle

\section{Introduction}\label{intro}

An important tool in the study of sums and extremes of sequences of regularly varying random variables $(X_{n})_{n}$ are the (functional) point processes of exceedences. They are defined as
$$ N_{n} = \sum_{i=1}^{n}\delta_{(i/n, a_{n}^{-1}X_{i})}, \qquad n \in \mathbb{N},$$
on $[0,1] \times (\mathbb{R} \setminus \{0\})$, where $(a_{n})_{n}$ is an appropriately chosen scaling sequence. Letting the sum go to infinity one obtains the corresponding point processes on $[0,\infty) \times (\mathbb{R} \setminus \{0\})$.
In the i.i.d.~case it is well known that the regular variation of the distribution of $X_{1}$ is equivalent to the convergence of point processes $N_{n}$ towards a Poisson point process (see for instance Resnick~\cite{Re87}, Proposition 3.21). By this result asymptotical distributional properties of various functionals of the sequence $(X_{n})_{n}$, such as partial sums, extremes, last-exit times and self-normalized sums, can be derived (see Resnick~\cite{Re86} and Leadbetter and Rootz\'{e}n~\cite{LeRo88}). There exists various extensions of this theory to dependent stationary time series, see for instance Davis and Resnick~\cite{DaRe85}, Davis and Hsing~\cite{DaHs95}, Basrak et al.~\cite{BKS} and Krizmani\'{c}~\cite{Kr14}. In~\cite{DaRe85} a stationary sequence of moving averages
\begin{equation}\label{e:linproc}
X_{n} = \sum_{j=0}^{\infty}c_{j}Z_{n-j}, \qquad n \in \mathbb{N},
\end{equation}
is studied, together with various functionals of $(X_{n})_{n}$ such as extremes, sums and sample covariance functions, where $(Z_{i})_{i \in \mathbb{Z}}$ is a sequence of real valued i.i.d.~random variables and $(c_{j})_{j \geq 0}$ is a sequence of real numbers such that the series in (\ref{e:linproc}) converges almost surely.
It is shown that if the $Z_{i}$'s have regularly varying tail probabilities, that is
\begin{equation}\label{e:regvarpr1}
 \Pr(|Z_{i}| > x) = x^{-\alpha}L(x), \qquad x >0,
\end{equation}
for some slowly varying function $L$ (i.e. $\lim_{x \to \infty}L(tx)/L(x) = 1$ for all $t>0$) and $\alpha >0$, and
\begin{equation}\label{e:pq}
  \lim_{x \to \infty} \frac{\Pr(Z_{i} > x)}{\Pr(|Z_{i}| > x)}=p, \qquad
    \lim_{x \to \infty} \frac{\Pr(Z_{i} < -x)}{\Pr(|Z_{i}| > x)}=q,
\end{equation}
where $p \in [0,1]$ and $p+q=1$, and the sequence of coefficients $(c_{j})_{j}$ satisfies
\begin{equation}\label{e:momcond}
\sum_{j=0}^{\infty} |c_{j}|^{\delta} < \infty \qquad \textrm{for some}  \ \delta < \alpha,\,0 < \delta \leq 1
\end{equation}
(under this condition the series in (\ref{e:linproc}) converges almost surely, see Cline~\cite{Cl83}, Theorem 2.1), then
\begin{equation}\label{e:ppconv}
 N_{n} \dto N = \sum_{i=1}^{\infty}\sum_{j=0}^{\infty}\delta_{(T_{i}, P_{i}c_{j})} \qquad \textrm{as} \ n \to \infty,
\end{equation}
in the space $\mathbf{M}_p([0,1] \times \mathbb{R} \setminus \{0\})$ of Radon point
measures on $[0,1] \times (\mathbb{R} \setminus \{0\})$,
where $N':=\sum_{i=1}^{\infty}\delta_{(T_{i}, P_{i})}$ is a Poisson point process with intensity measure $Leb \times \mu$, with
\begin{equation*}\label{e:mu}
  \mu(\rmd x) = \bigl( p \, 1_{(0, \infty)}(x) + q \, 1_{(-\infty, 0)}(x) \bigr) \, \alpha |x|^{-\alpha-1} \, \rmd x.
\end{equation*}
The scaling sequence $(a_{n})_{n}$ in the definition of the point processes $N_{n}$ consists of positive real numbers such that
\begin{equation}\label{eq:niz}
n \Pr(|Z_{1}| > a_{n}) \to 1 \qquad \textrm{as} \ n \to \infty.
\end{equation}
Niu~\cite{Ni97} extended this result to non-stationary moving averages of independent random variables with regularly varying tail probabilities and different scale parameters. More precisely, for moving averages
\begin{equation}\label{e:linprocNiu}
X_{n} = \sum_{j=0}^{\infty}c_{j}\sigma_{n-j}Z_{n-j}, \qquad n \in \mathbb{N},
\end{equation}
where the sequences $(Z_{i})$ and $(c_{j})$ are as above, and the scale parameters $\sigma_{i}$ satisfy $0 < \sigma_{i} \leq \delta_{0} < \infty$ and
$$ \frac{1}{n} \sum_{i=1}^{n}\sigma_{i}^{\alpha} \to \sigma^{\alpha} >0 \qquad \textrm{as} \ n \to \infty,$$
the point process convergence in (\ref{e:ppconv}) still holds, with $N'$ being a Poisson point process with intensity measure $Leb \times (\sigma^{\alpha} \mu)$ (Niu~\cite{Ni97}, Theorem 2.1). The non-stationarity of the series $(X_{n})$ is due to the nonconstant scale parameters $\sigma_{i}$, which in the limiting point process are represented by the ``average value'' $\sigma$.

In this article we will extend further this result to the case of weakly dependent innovations. We will assume the sequence $(Z_{i})$ satisfies the mixing condition $\mathcal{A}'$ (which holds under the more familiar strong mixing condition) and Leadbetter's dependence condition $D'$ familiar from extreme value analysis. Under similar assumptions as in~\cite{Ni97} we will derive the convergence of point processes $N_{n}$, and then using the continuous mapping theorem we will transfer this convergence to functional convergence of partial
maxima processes
in the space $D([0,1], \mathbb{R})$ of real valued c\`{a}dl\`{a}g functions on $[0,1]$. Due to possible clustering of large values, the most familiar of Skorokhod's topologies on $D([0,1], \mathbb{R})$, the $J_{1}$ topology, becomes inappropriate, and hence we will use the weaker Skorokhod $M_{1}$ topology.

 The paper is organized as follows. In Section~\ref{pointproc} we introduce basic notions and show preliminary results on point processes, and then prove a theorem about point process convergence in the case of non-stationary linear processes with weakly dependent heavy-tailed innovations. In Section~\ref{s:partialmp} we apply this theorem to obtain functional convergence of partial maxima processes
in the space of c\`{a}dl\`{a}g functions on $[0,1]$ with the Skorokhod $M_{1}$ topology. Finally, three technical results are given in Appendix A.

\section{Point process convergence}\label{pointproc}

  We start with some basic notions and results on point processes which will be used later on. For more background on the theory of point processes we refer to Kallenberg~\cite{Ka83} and Resnick~\cite{Re87}. Let $\mathbb{E} = [-\infty, \infty] \setminus \{0\}$. For $x,y \in \mathbb{E}$ define
 \begin{equation}\label{e:metricrho}
    \rho (x,y) = \max \Big\{ \Big| \frac{1}{|x|} -\frac{1}{|y|} \Big|, |\textrm{sign}\,x - \textrm{sign}\,y| \Big\},
 \end{equation}
where $\textrm{sign}\,z = z/|z|$. With the metric $\rho$, $\mathbb{E}$ becomes a locally compact, complete and separable matric space. A set $B \subseteq \mathbb{E}$ is relatively compact if it is bounded away from origin, that is, if there exists $u>0$ such that $B \subseteq \mathbb{E} \setminus [-u,u]$. Denote by $\mathcal{B}(\mathbb{E})$ the $\sigma$--algebra generated by $\rho$--open sets. Let $M_{+}(\mathbb{E})$ be the class of all Radon measures on $\mathbb{E}$, i.e.~all nonnegative measures that are finite on relatively compact subsets of $\mathbb{E}$. A useful topology for $M_{+}(\mathbb{E})$ is the vague topology which renders $M_{+}(\mathbb{E})$ a complete separable metric space. A sequence $(\mu_{n})_{n}$ in $M_{+}(\mathbb{E})$ converges vaguely to $\mu \in M_{+}(\mathbb{E})$ (written $\mu_{n} \xrightarrow{v} \mu$) if $\int f\,d \mu_{n} \to \int f\,d \mu$ for all $f \in C_{K}^{+}(\mathbb{E})$, where $C_{K}^{+}(\mathbb{E})$ denotes the class of all nonnegative continuous real functions on $\mathbb{E}$ with compact support. One metric that induces the vague topology is given by
\begin{equation}\label{e:vaguemetric}
 d_{v}(\mu_{1}, \mu_{2}) = \sum_{k=1}^{\infty} 2^{-k} \bigg( 1- \exp \bigg\{ - \bigg| \int_{\mathbb{E}} f_{k}(x)\,\mu_{1}(\rmd x) - \int_{\mathbb{E}} f_{k}\,\mu_{2}(\rmd x) \bigg| \bigg\} \bigg), \quad \mu_{1}, \mu_{2} \in M_{+}(\mathbb{E}),
\end{equation}
for some sequence of functions $(f_{k})_{k}$ in $C_{K}^{+}(\mathbb{E})$.

A Radon point measure is an element of $M_{+}(\mathbb{E})$ of the form $m = \sum_{i}\delta_{x_{i}}$, where $\delta_{x}$ is the Dirac measure. Denote by $M_{p}(\mathbb{E})$ the class of all Radon point measures. Since $M_{p}(\mathbb{E})$ is a subset of $M_{+}(\mathbb{E})$, we endow it with the relative topology. Let $\mathcal{M}_{p}(\mathbb{E})$ be the Borel $\sigma$--field of subsets of $M_{p}(\mathbb{E})$ generated by open sets. A point process on $\mathbb{E}$ is a measurable map from a given probability space to the measurable space $(M_{p}(\mathbb{E}), \mathcal{M}_{p}(\mathbb{E}))$. A standard example of point process is the Poisson process. Suppose $\mu$ is a given Radon measure on $\mathbb{E}$. Then $\xi$ is a Poisson process with intensity measure $\mu$, or synonymously, a Poisson random measure ($\mathrm{PRM}(\mu)$), if for all $A \in \mathcal{B}(\mathbb{E})$ and $k=0,1,\ldots$ it holds that
\begin{equation*}
 \mathrm{P}[\xi(A)=k] = \left\{
                          \begin{array}{cc}
                            \displaystyle \frac{[\mu(A)]^{k}}{k!} e^{-\mu(A)}, & \textrm{if} \ \mu(A) < \infty, \\[0.6em]
                            0, & \textrm{if} \ \mu(A) = \infty, \\
                          \end{array}
                        \right.
 \end{equation*}
 and if $A_{1}, \ldots, A_{k} \in \mathcal{B}(\mathbb{E})$ are mutually disjoint, then $\xi(A_{1}), \ldots, \xi(A_{k})$ are independent random variables.

A sequence of point processes $(\xi_{n})_{n}$ on $\mathbb{E}$ converges in distribution to a point process $\xi$ on $\mathbb{E}$ (written $\xi_{n} \xrightarrow{d} \xi$) if $\mathrm{E} f(\xi_{n}) \to \mathrm{E} f(\xi)$ for every bounded continuous function $f \colon M_{p}(\mathbb{E}) \to \mathbb{R}$. In dealing with distributions of point processes a very useful transform technique is the Laplace functional, since the point process convergence is characterized by convergence of Laplace functionals.
Denote by $\mathcal{B}_{+}$ the set of bounded measurable functions $f \colon \mathbb{E} \to [0,\infty)$. For a point process $\xi$ on $\mathbb{E}$ the Laplace functional of $\xi$ is the nonnegative function on $\mathcal{B}_{+}$ given by
 \begin{equation*}
  \Psi_{\xi}(f) = \mathrm{E} e^{-\xi(f)}, \qquad f \in \mathcal{B}_{+},
 \end{equation*}
 where $\xi(f) =  \int_{\mathbb{E}} f(x) \xi(\rmd x)$.
Then the Laplace functional of a point process $\xi$ uniquely determines the distribution of $\xi$, and given a sequence of point processes $(\xi_{n})_{n}$,
 \begin{equation}\label{e:ppLf}
    \xi_{n} \xrightarrow{d} \xi \qquad \textrm{iff} \qquad \Psi_{\xi_{n}}(f) \to \Psi_{\xi}(f) \quad \textrm{for all} \ f \in C_{K}^{+}(\mathbb{E})
 \end{equation}
  (see Kallenberg~\cite{Ka83}, Theorems 3.1 and 4.2). The Laplace functional of a Poisson process $\xi$ with intensity measure $\mu$ is of the form
\begin{equation}\label{e:LapPP}
   \Psi_{\xi}(f) = \exp \bigg\{ - \int_{\mathbb{E}} (1- e^{-f(x)})\,\mu(\rmd x) \bigg\}, \qquad f \in \mathcal{B}_{+}
\end{equation}
(see Resnick~\cite{Re87}, Proposition 3.6).

Now let $(Z_{i})_{i \in \mathbb{Z}}$ be a strictly stationary sequence of random variables with regularly varying tail probabilities as specified by (\ref{e:regvarpr1}) and (\ref{e:pq}). Let $(a_{n})_{n}$ be a sequence of positive real numbers tending to infinity such that
\begin{equation}\label{e:seqan}
n \Pr (|Z_{1}|>a_{n}) \to 1 \qquad \textrm{as} \ n \to \infty,
\end{equation}
for example we can take $a_{n}$ to be the $(1-n^{-1})$--quantile of the distribution function of $|Z_{1}|$. This and relation (\ref{e:regvarpr1}) then imply
\begin{equation}\label{e:seqan1}
 n \Pr(|Z_{1}| > u a_{n})=u^{-\alpha} \qquad \textrm{for all} \ u>0.
\end{equation}
Assume $(Z_{i})$ satisfies the local dependence condition $D'$:
\begin{equation}\label{e:D'cond}
 \lim_{k \to \infty} \limsup_{n \to \infty}~n \sum_{i=2}^{\lfloor n/k \rfloor} \Pr \bigg( \frac{|Z_{1}|}{a_{n}} > x, \frac{|Z_{i}|}{a_{n}} >x \bigg) = 0 \qquad \textrm{for all} \ x >0,
 \end{equation}
 where $\lfloor x \rfloor$ represents the greatest integer not larger than $x$.
 Although condition $D'$ is in general not satisfied even by $m$--dependent sequences, various stationary sequences satisfy (\ref{e:D'cond}), including processes which are instantaneous functions of stationary Gaussian processes with covariance function $r_{n}$ behaving like $r_{n} \ln n \to 0$ as $n \to \infty$ (see Davis~\cite{Da83}). Other examples of time series that satisfy condition $D'$, related to stochastic volatility models and ARMAX processes, can be found in Davis and Mikosch~\cite{DaMi09} and Ferreira and Canto e Castro~\cite{FeCa08}. This condition, together with the strong mixing property, assures that, as in the i.i.d.~case, the extremes of the sequence $(Z_{i})$ are isolated. This means that in this case the extremal index of the sequence $(Z_{i})$, which can be interpreted as the reciprocal mean cluster size of large exceedances, is equal to $1$. It is straightforward to see that condition $D'$ implies
 \begin{equation}\label{e:D'condrn}
 \lim_{n \to \infty}~n \sum_{i=2}^{r_{n}} \Pr \bigg( \frac{|Z_{1}|}{a_{n}} > x, \frac{|Z_{i}|}{a_{n}} >x \bigg) = 0 \qquad \textrm{for all} \ x >0,
 \end{equation}
 where $(r_{n})_{n}$ is any sequence of positive integers such that $r_{n} \to \infty$ and $r_{n}/n \to 0$ as $n \to \infty$.

Let $(\sigma_{i})_{i \in \mathbb{Z}}$ be a sequence of positive real numbers such that for some $\beta_{1}, \beta_{2} >0$
\begin{equation}\label{e:scaleparam1}
0 < \beta_{1} \leq \sigma_{i} \leq \beta_{2} < \infty \qquad \textrm{for all} \ i \in \mathbb{Z},
\end{equation}
and
\begin{equation}\label{e:scaleparam2}
\frac{1}{n} \sum_{i=1}^{n}\sigma_{i}^{\alpha} \to \sigma^{\alpha} >0 \qquad \textrm{as} \ n \to \infty.
\end{equation}
Set
\begin{equation}\label{e:defYi}
Y_{i} = \sigma_{i}Z_{i} \qquad \textrm{for all} \  i \in \mathbb{Z}.
\end{equation}
Assume there exists a sequence of positive integers $(r_{n})_{n}$ with $r_{n} \to \infty $ and $r_{n} / n \to 0$ as $n \to \infty$, such that the following mixing condition $\mathcal{A}'$ holds: for all $f \in C_{K}^{+}([0,1] \times \mathbb{E})$, as $n \to \infty$,
\begin{equation}\label{e:mixcon}
 \E \biggl[ \exp \biggl\{ - \sum_{i=1}^{n} f \biggl(\frac{i}{n}, \frac{Y_{i}}{a_{n}}
 \biggr) \biggr\} \biggr]
 - \prod_{k=1}^{k_{n}} \E \biggl[ \exp \biggl\{ - \sum_{i=1}^{r_{n}} f \biggl(\frac{kr_{n}}{n}, \frac{Y_{i+(k-1)r_{n}}}{a_{n}} \biggr) \biggr\} \biggr] \to 0,
\end{equation}
 where $k_{n} = \lfloor n / r_{n} \rfloor$. Condition $\mathcal{A}'$, roughly speaking, allows us to break the sequence $(Y_{i})$ into increasing and asymptotically independent blocks (cf.~Basrak et al.~\cite{BKS} for $\mathcal{A}'$ in the stationary case). This condition holds under the strong mixing property (see Lemma~\ref{l:A'strongmix} in Appendix A).

\begin{lem}\label{l:lemma1}
Let $(Z_{i})_{i \in \mathbb{Z}}$ be a strictly stationary sequence of random variables with regularly varying tail probabilities as specified by $(\ref{e:regvarpr1})$ and $(\ref{e:pq})$. Let $(\sigma_{i})_{i \in \mathbb{Z}}$ be a sequence of positive real numbers satisfying $(\ref{e:scaleparam1})$ and $(\ref{e:scaleparam2})$. Then
$$ \sum_{j=1}^{n} \delta_{j/n}(\,\cdot\,) \Pr ( a_{n}^{-1}Y_{j} \in\,\cdot\,) \vto Leb \times \nu \qquad \textrm{as} \ n \to \infty,$$
in $M_{+}([0,1] \times \mathbb{E})$,
 where
$$ \nu(\rmd x) = \sigma^{\alpha} \bigl( p \, 1_{(0, \infty)}(x) + q \, 1_{(-\infty, 0)}(x) \bigr) \, \alpha |x|^{-\alpha-1} \, \rmd x.$$
\end{lem}
\begin{proof}
Since $\mu_{n} \vto \mu$ iff $\mu_{n}(B) \to \mu(B)$ for all relatively compact sets $B$ such that $\mu(\partial B)=0$ (see Resnick~\cite{Re86}, Proposition 3.12), it suffices to show that, as $n \to \infty$,
\begin{equation}\label{e:conv1}
 \sum_{j=1}^{n} \delta_{j/n}([0, t)) \Pr ( a_{n}^{-1}Y_{j} \in (x, \infty]) \to Leb \times \nu\,([0,t) \times (x,\infty]) = t \nu((x,\infty])
\end{equation}
and
\begin{equation}\label{e:conv2}
 \sum_{j=1}^{n} \delta_{j/n}([0, t)) \Pr ( a_{n}^{-1}Y_{j} \in [-\infty, -x)) \to t \nu([-\infty, -x))
\end{equation}
for all $t \in (0,1]$ and $x>0$. We show only (\ref{e:conv1}) since (\ref{e:conv2}) can be shown similarly.

Assume first $p \neq 0$.
By stationarity of the sequence $(Z_{i})$ it holds that
\begin{equation}\label{e:aux1}
\sum_{j=1}^{n} \delta_{j/n}([0, t)) \Pr ( a_{n}^{-1}Y_{j} \in (x, \infty]) = \sum_{j=1}^{\floor{nt}} \Pr \Big(\frac{Y_{j}}{a_{n}} >x \Big) = \sum_{j=1}^{\floor{nt}} \Pr \Big( Z_{1} > \frac{a_{n}x}{\sigma_{j}} \Big),
\end{equation}
where $\floor{nt}$ denotes the integer part of $nt$.
Relations (\ref{e:regvarpr1}) and (\ref{e:pq}) imply the function
$$ f(y) = \Pr(Z_{1} > y)\,y^{\alpha} , \qquad y>0,$$
is slowly varying, since for $\lambda >0$
\begin{eqnarray*}
\lim_{y \to \infty} \frac{f(\lambda y)}{f(y)} &=& \lambda^{\alpha} \lim_{y \to \infty} \frac{\Pr(Z_{1} > \lambda y)}{\Pr(|Z_{1}|> \lambda y)} \cdot \frac{\Pr(|Z_{1}| > \lambda y)}{\Pr(|Z_{1}|> y)} \cdot \frac{\Pr(|Z_{1}| > y)}{\Pr(Z_{1}> y)}\\[0.4em]
 &=& \lambda^{\alpha} \cdot p \cdot \lambda^{-\alpha} \cdot \frac{1}{p}=1.
\end{eqnarray*}
Hence Theorem 1.5.2 in Bingham et al.~\cite{BiGoTe89} implies $f(\lambda y) /f(y) \to 1 $ (as $y \to \infty$) uniformly in $\lambda$ on each segment $[a,b]$ ($0 < a \leq b < \infty$). In particular
$$ F_{n}(\lambda) :=\frac{\Pr(Z_{1} > \lambda a_{n})\,\lambda^{\alpha}}{\Pr(Z_{1}>a_{n})} \to 1 \ \ (\textrm{as} \ n \to \infty) \qquad \textrm{uniformly in} \ \lambda \ \textrm{on} \ \Big[ \frac{x}{\beta_{2}}, \frac{x}{\beta_{1}} \Big],$$
from which, taking into account relation (\ref{e:scaleparam1}) we obtain
\begin{equation}\label{e:supinf}
 \sup_{j \geq 1} F_{n} \Big( \frac{x}{\sigma_{j}} \Big) \to 1 \quad \textrm{and} \quad \inf_{j \geq 1} F_{n} \Big( \frac{x}{\sigma_{j}} \Big) \to 1 \qquad \textrm{as} \ n \to \infty.
\end{equation}
Since
\begin{eqnarray*}
\sum_{j=1}^{\floor{nt}} \Pr \Big( Z_{1} > \frac{a_{n}x}{\sigma_{j}} \Big) &=& \sum_{j=1}^{\floor{nt}} F_{n} \Big(\frac{x}{\sigma_{j}}\Big) \Pr(Z_{1}>a_{n}) \Big( \frac{x}{\sigma_{j}}\Big)^{-\alpha}
  \leq x^{-\alpha} \sup_{j \geq 1} F_{n} \Big( \frac{x}{\sigma_{j}} \Big) \Pr(Z_{1}>a_{n})  \sum_{j=1}^{\floor{nt}}  \sigma_{j}^{\alpha}\\[0.4em]
 & = & x^{-\alpha}  \sup_{j \geq 1} F_{n} \Big( \frac{x}{\sigma_{j}} \Big) \frac{\Pr(Z_{1}>a_{n})}{\Pr(|Z_{1}|>a_{n})}\,n \Pr(|Z_{1}|>a_{n})  \frac{\floor{nt}}{n} \bigg( \frac{1}{\floor{nt}}\sum_{j=1}^{\floor{nt}}  \sigma_{j}^{\alpha} \bigg),
\end{eqnarray*}
 relations (\ref{e:supinf}), (\ref{e:pq}), (\ref{e:seqan}) and (\ref{e:scaleparam2}) yield
$$ \limsup_{n \to \infty} \sum_{j=1}^{\floor{nt}} \Pr \Big( Z_{1} > \frac{a_{n}x}{\sigma_{j}} \Big) \leq x^{-\alpha} p\,t \sigma^{\alpha},$$
and similarly
$$ \liminf_{n \to \infty} \sum_{j=1}^{\floor{nt}} \Pr \Big( Z_{1} > \frac{a_{n}x}{\sigma_{j}} \Big) \geq x^{-\alpha} p\,t \sigma^{\alpha}.$$
Therefore
$$ \lim_{n \to \infty} \sum_{j=1}^{\floor{nt}} \Pr \Big( Z_{1} > \frac{a_{n}x}{\sigma_{j}} \Big) = x^{-\alpha} p\,t \sigma^{\alpha} = t \nu ((x,\infty]).$$
This, together with (\ref{e:aux1}), yields (\ref{e:conv1}).

 In the case when $p=0$ note that by (\ref{e:scaleparam1}), (\ref{e:pq}) and (\ref{e:seqan1}) we have
 $$ \sum_{j=1}^{\floor{nt}} \Pr \Big( Z_{1} > \frac{a_{n}x}{\sigma_{j}} \Big) \leq \floor{nt} \Pr \Big( Z_{1} > \frac{a_{n}x}{\beta_{2}} \Big) \to 0 \quad \textrm{as} \ n \to \infty,$$
 yielding
 $$ \lim_{n \to \infty}  \sum_{j=1}^{\floor{nt}} \Pr \Big( Z_{1} > \frac{a_{n}x}{\sigma_{j}} \Big) = 0, $$
 and (\ref{e:conv1}) again holds (since in this case $\nu((x,\infty])=0$).
\end{proof}

\begin{lem}\label{l:lemma2}
Let $(Z_{i})_{i \in \mathbb{Z}}$ be a strictly stationary sequence of random variables satisfying condition $(\ref{e:D'cond})$ and with regularly varying tail probabilities as specified by $(\ref{e:regvarpr1})$ and $(\ref{e:pq})$. Let $(\sigma_{i})_{i \in \mathbb{Z}}$ be a sequence of positive real numbers satisfying $(\ref{e:scaleparam1})$ and $(\ref{e:scaleparam2})$. Assume the mixing condition $(\ref{e:mixcon})$ holds. Then for every $f \in C_{K}^{+}([0,1] \times \mathbb{E})$, as $n \to \infty$:
\begin{itemize}
\item[(i)] $\displaystyle \sup_{1 \leq k \leq k_{n}} \bigg| \E \biggl[ \exp \biggl\{ - \sum_{i=1}^{r_{n}} f \biggl(\frac{kr_{n}}{n}, \frac{Y_{i+(k-1)r_{n}}}{a_{n}} \biggr) \biggr\} -1 \biggr] \bigg| \to 0,$\\[0.6em]
\item[(ii)] $\displaystyle \sum_{k=1}^{k_{n}} \E \biggl[ \exp \biggl\{ - \sum_{i=1}^{r_{n}} f \biggl(\frac{kr_{n}}{n}, \frac{Y_{i+(k-1)r_{n}}}{a_{n}} \biggr) \biggr\} -1 \biggr] \to \iint_{[0,1] \times \mathbb{E}} (e^{-f(t,x)}-1)\,\rmd t \nu(\rmd x)$,
\end{itemize}
where $\nu$ is the measure defined in Lemma~$\ref{l:lemma1}$, and $(r_{n})_{n}$ and $(k_{n})_{n}$ are sequences that appear in condition $\mathcal{A}'$.
\end{lem}
\begin{proof}
(i) Take an arbitrary $f \in C_{K}^{+}([0,1] \times \mathbb{E})$. The support of $f$ is bounded away from origin, which implies
$f(t,x)=0$ for all $t \in [0,1]$ and $|x| \leq u$
 for some $u>0$. For all $n \in \mathbb{N}$ and $k=1,\ldots, k_{n}$ define
$$ T_{n,k}(f) =  \E \bigg[ \exp \biggl\{ - \sum_{i=1}^{r_{n}} f \biggl(\frac{kr_{n}}{n}, \frac{Y_{i+(k-1)r_{n}}}{a_{n}} \biggr) \biggr\} \bigg],$$
and
\begin{eqnarray*}
H_{1}^{n,k} &=& \bigg\{ \frac{|Y_{1+(k-1)r_{n}}|}{a_{n}} \leq u, \frac{|Y_{2+(k-1)r_{n}}|}{a_{n}} \leq u, \ldots, \frac{|Y_{r_{n}+(k-1)r_{n}}|}{a_{n}} \leq u \bigg\},\\[0.4em]
H_{2}^{n,k} &=& \bigg\{ \exists\,i \in \{1,\ldots,r_{n}\} \ \textrm{s.t.} \ \frac{|Y_{i+(k-1)r_{n}}|}{a_{n}} > u \ \textrm{and} \ \frac{|Y_{j+(k-1)r_{n}}|}{a_{n}} \leq u
\ \textrm{for all} \ j \in \{1,\ldots,r_{n}\} \setminus \{i\} \bigg\},\\[0.4em]
H_{3}^{n,k} &=& \bigg\{ \exists\,i, j \in \{1,\ldots,r_{n}\}, i \neq j, \ \textrm{s.t.} \ \frac{|Y_{i+(k-1)r_{n}}|}{a_{n}} > u \ \textrm{and} \
            \frac{|Y_{j+(k-1)r_{n}}|}{a_{n}} > u \bigg\}.
\end{eqnarray*}
Note that $H_{1}^{n,k}, H_{2}^{n,k}$ and $H_{3}^{n,k}$ are mutually disjoint events for which it holds that $1_{H_{1}^{n,k}} + 1_{H_{2}^{n,k}}+ 1_{H_{3}^{n,k}} =1$, where $1_{A}$ denotes the indicator function of a set $A$.
Denoting
$$H_{2}^{n,k}(i) =  \bigg\{\frac{|Y_{i+(k-1)r_{n}}|}{a_{n}} > u, \ \textrm{and} \ \frac{|Y_{j+(k-1)r_{n}}|}{a_{n}} \leq u \ \textrm{for all} \ j \in \{1,\ldots,r_{n}\} \setminus \{i\} \bigg\},$$
it holds that $H_{2}^{n,k}$ is a disjoint union of the sets $H_{2}^{n,k}(i)$ ($i=1,\ldots,r_{n}$), and hence we have\\[-0.4em]
\begin{eqnarray*}
\nonumber  T_{n,k}(f) -1\\[0.2em]
 \nonumber & \hspace*{-11em} =&  \hspace*{-5.5em} \E \biggl[ 1_{H_{1}^{n,k}} + \sum_{i=1}^{r_{n}} \exp  \biggl\{ - f \biggl(\frac{kr_{n}}{n}, \frac{Y_{i+(k-1)r_{n}}}{a_{n}} \biggr)
                \biggr\} 1_{H_{2}^{n,k}(i)}
  + \exp \biggl\{ - \sum_{i=1}^{r_{n}} f \biggl(\frac{kr_{n}}{n}, \frac{Y_{i+(k-1)r_{n}}}{a_{n}} \biggr) \biggr\} 1_{H_{3}^{n,k}} -1 \biggr]\\[0.5em]
 \nonumber &\hspace*{-11em} =& \hspace*{-5.5em} \sum_{i=1}^{r_{n}} \E \biggl[  \bigg( \exp \biggl\{ - f \biggl(\frac{kr_{n}}{n}, \frac{Y_{i+(k-1)r_{n}}}{a_{n}} \biggr)
                \biggr\} -1 \bigg) 1_{H_{2}^{n,k}(i)} \bigg]
            + \E \biggl[ \bigg( \exp \biggl\{ - \sum_{i=1}^{r_{n}} f \biggl(\frac{kr_{n}}{n}, \frac{Y_{i+(k-1)r_{n}}}{a_{n}} \biggr) \biggr\} -1 \bigg) 1_{H_{3}^{n,k}} \biggr].
\end{eqnarray*}
Define
$$ W_{1}^{n,k,i}(f) =   \exp \biggl\{ - f \biggl(\frac{kr_{n}}{n}, \frac{Y_{i+(k-1)r_{n}}}{a_{n}} \biggr)
                \biggr\} -1,$$
and
$$W_{2}^{n,k}(f)=  \exp \biggl\{ - \sum_{i=1}^{r_{n}} f \biggl(\frac{kr_{n}}{n}, \frac{Y_{i+(k-1)r_{n}}}{a_{n}} \biggr) \biggr\} -1 .$$
Then
\begin{equation}\label{e:Tn1}
  T_{n,k}(f) -1 = \sum_{i=1}^{r_{n}} \E \Big[ W_{1}^{n,k,i}(f)1_{H_{2}^{n,k}(i)} \Big] + \E \Big[ W_{2}^{n,k}(f) 1_{H_{3}^{n,k}} \Big].
\end{equation}
Note that
\begin{equation}\label{e:Tn2}
 1_{H_{2}^{n,k}(i)} = 1_{\big\{ \frac{|Y_{i}+(k-1)r_{n}|}{a_{n}} > u \big\}} - 1_{\widetilde{H}_{2}^{n,k}(i)} = 1 -  1_{\big\{ \frac{|Y_{i}+(k-1)r_{n}|}{a_{n}} \leq u \big\}} - 1_{\widetilde{H}_{2}^{n,k}(i)},
\end{equation}
where
$$\widetilde{H}_{2}^{n,k}(i) =  \bigg\{\frac{|Y_{i+(k-1)r_{n}}|}{a_{n}} > u, \ \textrm{and there} \ \exists \ j \in \{1,\ldots, r_{n}\} \setminus \{i\} \ \textrm{s.t.} \ \frac{|Y_{j+(k-1)r_{n}}|}{a_{n}} > u \bigg\}.$$
Then
\begin{equation}\label{e:Tn3}
 \E \Big[ W_{1}^{n,k,i}(f) 1_{\big\{ \frac{|Y_{i}+(k-1)r_{n}|}{a_{n}} \leq u \big\}} \Big]=0,
\end{equation}
and
\begin{equation}\label{e:Tn4}
 \Big| \E \Big[  W_{1}^{n,k,i}(f) 1_{\widetilde{H}_{2}^{n,k}(i)} \Big] \Big| \leq 2 \Pr(\widetilde{H}_{2}^{n,k}(i)),
\end{equation}
where we used the fact that for nonnegative function $f$ it holds that $|e^{-f(t,x)}-1| \leq 2$ for all $t \in [0,1]$ and $x \in \mathbb{E}$.
Similarly
\begin{equation}\label{e:Tn5}
 \Big| \E \Big[  W_{2}^{n,k}(f) 1_{H_{3}^{n,k}} \Big] \Big| \leq 2 \Pr(H_{3}^{n,k}).
\end{equation}
From the definition of the event $\widetilde{H}_{2}^{n,k}(i)$ and taking into account condition (\ref{e:scaleparam1}) it follows that
\begin{eqnarray*}
 \Pr(\widetilde{H}_{2}^{n,k}(i)) &\leq& \sum_{\scriptsize \begin{array}{l}
                          j=1  \\[0.1em]
                          j \neq i
                        \end{array}}^{r_{n}} \Pr \bigg( \frac{|Y_{i+(k-1)r_{n}}|}{a_{n}} > u, \frac{|Y_{j+(k-1)r_{n}}|}{a_{n}} >u  \bigg)\\[0.4em]
   &=& \sum_{\scriptsize \begin{array}{l}
                          j=1  \\[0.1em]
                          j \neq i
                        \end{array}}^{r_{n}} \Pr \bigg( \frac{\sigma_{i+(k-1)r_{n}}|Z_{i+(k-1)r_{n}}|}{a_{n}} > u, \frac{\sigma_{j+(k-1)r_{n}}|Z_{j+(k-1)r_{n}}|}{a_{n}} >u  \bigg)\\[0.4em]
   & \leq & \sum_{\scriptsize \begin{array}{l}
                          j=1  \\[0.1em]
                          j \neq i
                        \end{array}}^{r_{n}} \Pr \bigg( \frac{|Z_{i+(k-1)r_{n}}|}{a_{n}} > \frac{u}{\beta_{2}}, \frac{|Z_{j+(k-1)r_{n}}|}{a_{n}} > \frac{u}{\beta_{2}}  \bigg),
\end{eqnarray*}
and by stationarity of the sequence $(Z_{i})$ we get
\begin{equation}\label{e:Tn3-i}
 \sum_{i=1}^{r_{n}}  \Pr(\widetilde{H}_{2}^{n,k}(i)) \leq \frac{2r_{n}}{n}\,n \sum_{j=2}^{r_{n}} \Pr \bigg( \frac{|Z_{1}|}{a_{n}} > \frac{u}{\beta_{2}}, \frac{|Z_{j}|}{a_{n}} > \frac{u}{\beta_{2}}  \bigg).
\end{equation}
Note that the quantity on the right-hand side in the previous relation does not depend on $k$. This, together with condition (\ref{e:D'cond}) (i.e.~relation (\ref{e:D'condrn})) and the fact that $r_{n}/n \to 0$ as $n \to \infty$ yields
\begin{equation*}
 \lim_{n \to \infty} \sum_{i=1}^{r_{n}}  \Pr(\widetilde{H}_{2}^{n,k}(i)) = 0 \qquad \textrm{uniformly in} \ k \in \{1,\ldots,k_{n}\}.
\end{equation*}
Therefore by (\ref{e:Tn4})
\begin{equation}\label{e:Tn6}
 \lim_{n \to \infty} \sum_{i=1}^{r_{n}}  \E \Big[  W_{1}^{n,k,i}(f) 1_{\widetilde{H}_{2}^{n,k}(i)} \Big]  = 0 \qquad \textrm{uniformly in} \ k \in \{1,\ldots,k_{n}\}.
\end{equation}
Similarly we have
\begin{equation*}
  \Pr(H_{3}^{n,k}) \leq   \sum_{i=1}^{r_{n}} \sum_{j=i+1}^{r_{n}}  \Pr \bigg( \frac{|Z_{i}|}{a_{n}} > \frac{u}{\beta_{2}}, \frac{|Z_{j}|}{a_{n}} > \frac{u}{\beta_{2}}  \bigg)
  \leq r_{n} \sum_{j=2}^{r_{n}} \Pr \bigg( \frac{|Z_{1}|}{a_{n}} > \frac{u}{\beta_{2}}, \frac{|Z_{j}|}{a_{n}} > \frac{u}{\beta_{2}}  \bigg),
\end{equation*}
and consequently by (\ref{e:Tn5})
\begin{equation}\label{e:Tn7}
 \lim_{n \to \infty}  \E \Big[  W_{2}^{n,k}(f) 1_{H_{3}^{n,k}} \Big] =0 \qquad \textrm{uniformly in} \ k \in \{1,\ldots,k_{n}\}.
\end{equation}
Therefore from (\ref{e:Tn1}) taking into account relations (\ref{e:Tn2}), (\ref{e:Tn3}), (\ref{e:Tn6}) and (\ref{e:Tn7}) we obtain
\begin{eqnarray}\label{e:Tn8}
  \nonumber T_{n,k}(f)-1 &=&  \sum_{i=1}^{r_{n}} \E \Big[ W_{1}^{n,k,i}(f) \Big] + o(1)\\[0.5em]
      &=&  \sum_{i=1}^{r_{n}} \E \biggl[ \exp \biggl\{ - f \biggl(\frac{kr_{n}}{n}, \frac{Y_{i+(k-1)r_{n}}}{a_{n}} \biggr) \biggr\} -1 \biggr] + o(1)
\end{eqnarray}
as $n \to \infty$ uniformly in $k \in \{1,\ldots,k_{n}\}$,
where $b_{n} = c_{n} + o(d_{n})$ as $n \to \infty$ means that $\lim_{n \to \infty}(b_{n}-c_{n}) /d_{n} = 0$.
Observe that
$$ \E \biggl[ \exp \biggl\{ - f \biggl(\frac{kr_{n}}{n}, \frac{Y_{i+(k-1)r_{n}}}{a_{n}} \biggr) \biggr\} \biggr] =  \int_{\mathbb{E}} e^{-f(kr_{n}/n,x)} \Pr \bigg( \frac{Y_{i+(k-1)r_{n}}}{a_{n}} \in \rmd x \bigg),$$
and hence
\begin{eqnarray*}
 \Big| \sum_{i=1}^{r_{n}} \E \Big[ W_{1}^{n,k,i}(f) \Big] \Big| &\leq&  \sum_{i=1}^{r_{n}} \int_{\mathbb{E}} |e^{-f(kr_{n}/n,x)}-1| \Pr \bigg( \frac{Y_{i+(k-1)r_{n}}}{a_{n}} \in \rmd x \bigg)\\[0.4em]
 &=& \sum_{i=1}^{r_{n}} \int_{|x| >u} |e^{-f(kr_{n}/n,x)}-1| \Pr \bigg( \frac{Y_{i+(k-1)r_{n}}}{a_{n}} \in \rmd x \bigg)\\[0.4em]
 & \leq & 2 \sum_{i=1}^{r_{n}} \Pr \bigg( \frac{|Y_{i+(k-1)r_{n}}|}{a_{n}} > u \bigg)
  \leq  2r_{n} \Pr \bigg( \frac{|Z_{1}|}{a_{n}} > \frac{u}{\beta_{2}} \bigg)\\[0.4em]
   &=& \frac{2r_{n}}{n}\,n \Pr \bigg( \frac{|Z_{1}|}{a_{n}} > \frac{u}{\beta_{2}} \bigg).
\end{eqnarray*}
From this, by relation (\ref{e:seqan1}) and the fact that $r_{n}/n \to 0$ as $n \to \infty$, it follows that
\begin{equation*}
\lim_{n \to \infty} \sum_{i=1}^{r_{n}} \E \Big[ W_{1}^{n,k,i}(f) \Big]=0 \qquad \textrm{uniformly in} \ k \in \{1,\ldots,k_{n}\}.
\end{equation*}
Therefore from (\ref{e:Tn8}) we obtain $T_{n,k}(f)-1 \to 0$ as $n \to \infty$ uniformly in $k$, i.e.
$$ \sup_{1 \leq k \leq k_{n}} \bigg| \E \biggl[ \exp \biggl\{ - \sum_{i=1}^{r_{n}} f \biggl(\frac{kr_{n}}{n}, \frac{Y_{i+(k-1)r_{n}}}{a_{n}} \biggr) \biggr\} -1 \biggr] \bigg| \to 0 \qquad \textrm{as} \ n \to \infty.\\[0.7em]$$

(ii) Take again an arbitrary $f \in C_{K}^{+}([0,1] \times \mathbb{E})$ with $f(t,x)=0$ for all $t \in [0,1]$ and $|x| \leq u$
 for some $u>0$. By (\ref{e:Tn1}) we have
 \begin{equation}\label{e:Tn-ii1}
  \sum_{k=1}^{k_{n}}(T_{n,k}(f) -1) = \sum_{k=1}^{k_{n}} \sum_{i=1}^{r_{n}} \E \Big[ W_{1}^{n,k,i}(f)1_{H_{2}^{n,k}(i)} \Big] + \sum_{k=1}^{k_{n}} \E \Big[ W_{2}^{n,k}(f) 1_{H_{3}^{n,k}} \Big].
\end{equation}
Relation (\ref{e:Tn3-i}), together with (\ref{e:D'condrn}) and the fact that $k_{n}r_{n} \leq n$, yields
$$\sum_{k=1}^{k_{n}} \sum_{i=1}^{r_{n}}  \Pr(\widetilde{H}_{2}^{n,k}(i)) \leq 2k_{n}r_{n} \sum_{j=2}^{r_{n}} \Pr \bigg( \frac{|Z_{1}|}{a_{n}} > \frac{u}{\beta_{2}}, \frac{|Z_{j}|}{a_{n}} > \frac{u}{\beta_{2}}  \bigg) \to 0 \qquad \textrm{as} \ n \to \infty,$$
and hence by (\ref{e:Tn4}) we get
\begin{equation}\label{e:Tn-ii2}
 \lim_{n \to \infty}   \sum_{k=1}^{k_{n}} \sum_{i=1}^{r_{n}} \E \Big[  W_{1}^{n,k,i}(f) 1_{\widetilde{H}_{2}^{n,k}(i)} \Big]  = 0.
\end{equation}
Similarly we have
\begin{equation*}
  \sum_{k=1}^{k_{n}}\Pr(H_{3}^{n,k}) \leq k_{n} r_{n} \sum_{j=2}^{r_{n}} \Pr \bigg( \frac{|Z_{1}|}{a_{n}} > \frac{u}{\beta_{2}}, \frac{|Z_{j}|}{a_{n}} > \frac{u}{\beta_{2}}  \bigg) \to 0 \qquad \textrm{as} \ n \to \infty,
\end{equation*}
and consequently by (\ref{e:Tn5})
\begin{equation}\label{e:Tn-ii3}
 \lim_{n \to \infty}  \sum_{k=1}^{k_{n}} \E \Big[  W_{2}^{n,k}(f) 1_{H_{3}^{n,k}} \Big] =0.
\end{equation}
Thus from (\ref{e:Tn-ii1}), taking into account relations (\ref{e:Tn2}), (\ref{e:Tn3}), (\ref{e:Tn-ii2}) and (\ref{e:Tn-ii3}), similarly as in (i) we obtain
\begin{eqnarray*}
  \nonumber \sum_{k=1}^{k_{n}} (T_{n,k}(f)-1) &=&  \sum_{k=1}^{k_{n}} \sum_{i=1}^{r_{n}} \E \Big[ W_{1}^{n,k,i}(f) \Big] + o(1)\\[0.5em]
      & =&  \sum_{k=1}^{k_{n}} \sum_{i=1}^{r_{n}} \E \biggl[ \exp \biggl\{ - f \biggl(\frac{kr_{n}}{n}, \frac{Y_{i+(k-1)r_{n}}}{a_{n}} \biggr) \biggr\} -1 \biggr] + o(1)
\end{eqnarray*}
as $n \to \infty$. Note that
\begin{equation}\label{e:Tn-ii4}
   \sum_{k=1}^{k_{n}} (T_{n,k}(f)-1)
     =G_{n} + \sum_{k=1}^{k_{n}} \sum_{i=1}^{r_{n}} \int_{\mathbb{E}} \Big(e^{-f((i+(k-1)r_{n})/n,x)} -1 \Big) \Pr \bigg( \frac{Y_{i+(k-1)r_{n}}}{a_{n}} \in \rmd x \bigg) + o(1),
\end{equation}
where
$$ G_{n} = \sum_{k=1}^{k_{n}} \sum_{i=1}^{r_{n}}  \int_{\mathbb{E}} \Big(e^{-f(kr_{n}/n,x)} - e^{-f((i+(k-1)r_{n})/n,x)} \Big) \Pr \bigg( \frac{Y_{i+(k-1)r_{n}}}{a_{n}} \in \rmd x \bigg).$$
Since for any $t>0$ there exists a constant $C_{t} >0$ such that $|1-e^{-x}| \leq C_{t}|x|$ for all $|x| \leq t$, we have
\begin{eqnarray*}
 |G_{n}| &  \leq & \sum_{k=1}^{k_{n}} \sum_{i=1}^{r_{n}} \int_{\mathbb{E}} \Big| e^{-f(kr_{n}/n,x)} (1 - e^{f(kr_{n}/n,x) - f((i+(k-1)r_{n})/n,x)}) \Big| \Pr \bigg( \frac{Y_{i+(k-1)r_{n}}}{a_{n}} \in \rmd x \bigg)\\[0.4em]
 & \leq &  C_{2M} \sum_{k=1}^{k_{n}} \sum_{i=1}^{r_{n}} \int_{\mathbb{E}} \bigg| f \bigg( \frac{kr_{n}}{n},x \bigg) -  f \bigg( \frac{i+(k-1)r_{n}}{n},x \bigg) \bigg| \Pr \bigg( \frac{Y_{i+(k-1)r_{n}}}{a_{n}} \in \rmd x \bigg),\\[0.4em]
 & = &  C_{2M} \sum_{k=1}^{k_{n}} \sum_{i=1}^{r_{n}} \int_{|x| >u} \bigg| f \bigg( \frac{kr_{n}}{n},x \bigg) - f \bigg( \frac{i+(k-1)r_{n}}{n},x \bigg) \bigg| \Pr \bigg( \frac{Y_{i+(k-1)r_{n}}}{a_{n}} \in \rmd x \bigg)
\end{eqnarray*}
where $M$ is some positive real number such that $|f(t,x)| \leq M$ for all $t \in [0,1]$ and $x \in \mathbb{E}$, and the last equality holds since $f(t,x)=0$ for all $t \in [0,1]$ and $|x| \leq u$. Since $f$ is continuous on a compact set $[0,1] \times \{x \in \mathbb{E} : |x| \geq u \}$, it is also uniformly continuous, and hence $\omega_{f,u}(\epsilon) \to 0$ as $\epsilon \to 0$, where
$$ \omega_{f,u}(\epsilon) = \sup \{ |f(t,x)-f(s,y)| : t,s \in [0,1],\,x,y \in \mathbb{E} \setminus (-u,u) \ \textrm{and} \ d((t,x),(s,y)) \leq \epsilon \}$$
is the modulus of continuity of the function $f$ restricted to $[0,1] \times (\mathbb{E} \setminus (-u,u))$, and $d$ is the metric on the direct product of metric spaces $[0,1]$ and $\mathbb{E}$ given by
$$d ((t,x), (s,y)) = \max \{ |t-s|, \rho(x,y) \},$$
with $\rho$ defined in (\ref{e:metricrho}). Therefore
\begin{eqnarray*}
 |G_{n}| &\leq& C_{2M}\,\omega_{f,u} \bigg( \sup_{1\leq k \leq k_{n}} \sup_{1 \leq i \leq r_{n}} \bigg| \frac{kr_{n}}{n} - \frac{i+(k-1)r_{n}}{n} \bigg| \bigg) \sum_{k=1}^{k_{n}} \sum_{i=1}^{r_{n}} \Pr \bigg( \frac{|Y_{i+(k-1)r_{n}}|}{a_{n}} >u \bigg)\\[0.4em]
         &=& C_{2M}\,\omega_{f,u} \bigg(  \sup_{1 \leq i \leq r_{n}} \frac{|r_{n}-i|}{n} \bigg) \sum_{k=1}^{k_{n}} \sum_{i=1}^{r_{n}} \Pr \bigg( \frac{\sigma_{i+(k-1)r_{n}}|Z_{i+(k-1)r_{n}}|}{a_{n}} >u \bigg)\\[0.4em]
         &\leq & C_{2M}\,\omega_{f,u} \bigg(  \sup_{1 \leq i \leq r_{n}} \frac{|r_{n}-i|}{n} \bigg) k_{n}r_{n} \Pr \bigg( \frac{|Z_{1}|}{a_{n}} >\frac{u}{\beta_{2}} \bigg).
\end{eqnarray*}
Since
$$ \sup_{1 \leq i \leq r_{n}} \frac{|r_{n}-i|}{n} \leq \frac{r_{n}}{n} \to 0 \qquad \textrm{as} \ n \to \infty,$$
it follows that
$$\omega_{f,u} \bigg(  \sup_{1 \leq i \leq r_{n}} \frac{|r_{n}-i|}{n} \bigg) \to 0 \qquad \textrm{as} \ n \to \infty,$$
and this together with
$$ \limsup_{n \to \infty} k_{n}r_{n} \Pr \bigg( \frac{|Z_{1}|}{a_{n}} >\frac{u}{\beta_{2}} \bigg) \leq \limsup_{n \to \infty} n \Pr \bigg( \frac{|Z_{1}|}{a_{n}} >\frac{u}{\beta_{2}} \bigg) = \Big( \frac{u}{\beta_{2}} \Big)^{-\alpha}$$
(which holds by relation (\ref{e:seqan1}))
yields $G_{n} \to 0$ as $n \to \infty$. Hence by (\ref{e:Tn-ii4}) we have
\begin{equation}\label{e:Tn-ii5}
 \sum_{k=1}^{k_{n}} (T_{n,k}(f)-1) = \sum_{k=1}^{k_{n}} \sum_{i=1}^{r_{n}} \int_{\mathbb{E}} (e^{-f((i+(k-1)r_{n})/n,x)}-1) \Pr \bigg( \frac{Y_{i+(k-1)r_{n}}}{a_{n}} \in \rmd x \bigg) + o(1)
\end{equation}
as $n \to \infty$.
Note that
\begin{eqnarray*}
 \sum_{k=1}^{k_{n}} \sum_{i=1}^{r_{n}} \int_{\mathbb{E}} (e^{-f((i+(k-1)r_{n})/n,x)}-1) \Pr \bigg( \frac{Y_{i+(k-1)r_{n}}}{a_{n}} \in \rmd x \bigg)
 = \sum_{j=1}^{k_{n}r_{n}} \int_{\mathbb{E}} (e^{-f(j/n,x)}-1) \Pr \bigg( \frac{Y_{j}}{a_{n}} \in \rmd x \bigg)
\end{eqnarray*}
and
\begin{eqnarray*}
\bigg| \sum_{j=1}^{k_{n}r_{n}} \int_{\mathbb{E}} (e^{-f(j/n,x)}-1) \Pr \bigg( \frac{Y_{j}}{a_{n}} \in \rmd x \bigg) - \sum_{j=1}^{n} \int_{\mathbb{E}} (e^{-f(j/n,x)}-1) \Pr \bigg( \frac{Y_{j}}{a_{n}} \in \rmd x \bigg) \bigg| &&\\[0.4em]
  & \hspace*{-64em} \leq& \hspace*{-32em} \sum_{j=k_{n}r_{n}+1}^{n} \int_{\mathbb{E}} |e^{-f(j/n,x)}-1| \Pr \bigg( \frac{Y_{j}}{a_{n}} \in \rmd x \bigg)
  = \sum_{j=k_{n}r_{n}+1}^{n} \int_{|x| > u} |e^{-f(j/n,x)}-1| \Pr \bigg( \frac{Y_{j}}{a_{n}} \in \rmd x \bigg)\\[0.4em]
  & \hspace*{-64em} \leq& \hspace*{-32em} 2 \sum_{j=k_{n}r_{n}+1}^{n} \Pr \bigg( \frac{|Z_{j}|}{a_{n}} > \frac{u}{\beta_{2}} \bigg) = 2 (n-k_{n}r_{n})  \Pr \bigg( \frac{|Z_{1}|}{a_{n}} > \frac{u}{\beta_{2}} \bigg)\\[0.4em]
   & \hspace*{-64em} =& \hspace*{-32em}2 \Big( 1- \frac{k_{n}r_{n}}{n} \Big)\,n \Pr \bigg( \frac{|Z_{1}|}{a_{n}} > \frac{u}{\beta_{2}} \bigg) \to 0 \qquad \textrm{as} \ n \to \infty,
\end{eqnarray*}
since $1-k_{n}r_{n}/n \to 0$ and $n \Pr( |Z_{1}| > ua_{n}/\beta_{2}) \to (u/ \beta_{2})^{-\alpha}$ as $n \to \infty$. Hence by (\ref{e:Tn-ii5}) we have
 \begin{equation}\label{e:Tn10}
 \sum_{k=1}^{k_{n}} (T_{n,k}(f)-1) = \sum_{j=1}^{n} \int_{\mathbb{E}} (e^{-f(j/n,x)}-1) \Pr \bigg( \frac{Y_{j}}{a_{n}} \in \rmd x \bigg) + o(1) \quad \textrm{as} \ n \to \infty.
\end{equation}
Now, let
$$\lambda_{n}(\rmd t, \rmd x) = \sum_{j=1}^{n} \delta_{j/n}(\rmd t) \Pr ( a_{n}^{-1}Y_{j} \in \rmd x), \qquad n \in \mathbb{N}.$$
Then by Lemma~\ref{l:lemma1}, $\lambda_{n} \vto Leb \times \nu$ as $n \to \infty$, and this yields
\begin{eqnarray*}
 \sum_{j=1}^{n} \int_{\mathbb{E}} (e^{-f(j/n,x)}-1) \Pr \bigg( \frac{Y_{j}}{a_{n}} \in \rmd x \bigg) &=& \iint_{[0,1] \times \mathbb{E}}(e^{-f(t,x)}-1)\,\lambda_{n}(\rmd t, \rmd x)\\[0.4em]
 &\to&  \iint_{[0,1] \times \mathbb{E}} (e^{-f(t,x)}-1)\,\rmd t \nu(\rmd x),
\end{eqnarray*}
as $n \to \infty$. Therefore from (\ref{e:Tn10}) we obtain
\begin{equation*}
 \lim_{n \to \infty} \sum_{k=1}^{k_{n}} (T_{n,k}(f)-1) = \iint_{[0,1] \times \mathbb{E}} (e^{-f(t,x)}-1)\,\rmd t \nu(\rmd x),
\end{equation*}
and this concludes the proof.
\end{proof}

Let $(c_{j})_{j \geq 0}$ be a sequence of real numbers satisfying condition (\ref{e:momcond}).
For linear processes
\begin{equation}\label{e:linproc1}
X_{n} = \sum_{j=0}^{\infty}c_{j}Y_{n-j}, \qquad n \in \mathbb{N},
\end{equation}
with $Y_{i}$ defined in (\ref{e:defYi}),
define time-space point processes
\begin{equation}\label{e:ppNn}
N_{n} = \sum_{i=1}^{n}\delta_{(i/n, a_{n}^{-1}X_{i})}, \qquad n \in \mathbb{N}.
\end{equation}
Before proving the main result in this section on convergence of point processes $N_{n}$, we will prove the point process convergence in a special case when $c_{0}=1$ and $c_{j}=0$ for all $j \in \mathbb{N}$. The next result will be used in proving the general point process convergence, but it is also of independent interest. Let
\begin{equation}\label{e:ppN*n}
N^{*}_{n} = \sum_{i=1}^{n}\delta_{(i/n, a_{n}^{-1}Y_{i})}, \qquad n \in \mathbb{N}.
\end{equation}

\begin{prop}\label{p:ppconvY}
Let $(Z_{i})_{i \in \mathbb{Z}}$ be a strictly stationary sequence of random variables satisfying condition $(\ref{e:D'cond})$ and with regularly varying tail probabilities as specified by $(\ref{e:regvarpr1})$ and $(\ref{e:pq})$. Let $(\sigma_{i})_{i \in \mathbb{Z}}$ be a sequence of positive real numbers satisfying $(\ref{e:scaleparam1})$ and $(\ref{e:scaleparam2})$. Assume the mixing condition $(\ref{e:mixcon})$ holds. Then
$$ N^{*}_{n} \dto N^{*} := \sum_{i=1}^{\infty}\delta_{(T_{i}, P_{i})} \qquad \textrm{as} \ n \to \infty,$$
in $M_p([0,1] \times \mathbb{E})$, where $N^{*}$ is a Poisson process with intensity measure $Leb \times \nu$, with $\nu$ as in Lemma~$\ref{l:lemma1}$.
\end{prop}
\begin{proof}
Let $(Y_{k,j})_{j \in \mathbb{N}}$, with $k \in \mathbb{N}$, be independent copies of $(Y_{j})_{j \in \mathbb{N}}$, and define
$$ \widehat{N}_{n} = \sum_{k=1}^{k_{n}} \sum_{i=1}^{r_{n}}  \delta_{(kr_{n}/n, a_{n}^{-1}Y_{k, i+(k-1)r_{n}})}, \qquad n \in \mathbb{N},$$
where $(r_{n})_{n}$ and $(k_{n})_{n}$ are sequences that appear in condition $\mathcal{A}'$.
Since for every $f \in C_{K}^{+}([0,1] \times \mathbb{E})$,
$$ \Psi_{N_{n}^{*}}(f) = \E e^{-N_{n}^{*}(f)} = \E \biggl[ \exp \biggl\{ - \sum_{i=1}^{n} f \biggl(\frac{i}{n}, \frac{Y_{i}}{a_{n}}
 \biggr) \biggr\} \biggr]$$
  and
\begin{eqnarray}\label{e:PsiNtilda}
\nonumber \Psi_{\widehat{N}_{n}}(f) &=& \E \biggl[ \exp \biggl\{ - \sum_{k=1}^{k_{n}} \sum_{i=1}^{r_{n}} f \biggl(\frac{kr_{n}}{n}, \frac{Y_{k,i+(k-1)r_{n}}}{a_{n}} \biggr) \bigg\} \bigg]\\[0.2em]
  &=&
\prod_{k=1}^{k_{n}} \E \biggl[ \exp \biggl\{ - \sum_{i=1}^{r_{n}} f \biggl(\frac{kr_{n}}{n}, \frac{Y_{i+(k-1)r_{n}}}{a_{n}} \biggr) \biggr\} \biggr],
\end{eqnarray}
condition $\mathcal{A}'$ implies
\begin{equation}\label{e:diffppconv}
 \Psi_{N_{n}^{*}}(f) - \Psi_{\widehat{N}_{n}}(f) \to 0 \qquad \textrm{as} \ n \to \infty,
\end{equation}
and hence by relation (\ref{e:ppLf}) we conclude that the weak limits of $N_{n}^{*}$ and $\widehat{N}_{n}$, as $n \to \infty$, must coincide.

As in the proof of Lemma~\ref{l:lemma2}, for all $n \in \mathbb{N}$ and $k=1,\ldots, k_{n}$ define
$$ T_{n,k}(f) = \E \biggl[ \exp \biggl\{ - \sum_{i=1}^{r_{n}} f \biggl(\frac{kr_{n}}{n}, \frac{Y_{i+(k-1)r_{n}}}{a_{n}} \biggr) \biggr\} \biggr].$$
Then by Lemma~\ref{l:lemma2} (i) the inequality $|T_{n,k}(f)-1| \leq 1/2$ holds for large $n$ uniformly in $k \in \{1,\ldots,k_{n}\}$.
Using the elementary inequality
$$ |\ln (1+x) -x| \leq x^{2} \qquad \textrm{for} \ |x| \leq \frac{1}{2}$$
(which is derived from the Taylor series of the function $x \mapsto \ln (1+x)$), we obtain
\begin{equation*}\label{e:Ts1}
|\ln T_{n,k}(f) - (T_{n,k}(f)-1)| \leq (T_{n,k}(f)-1)^{2}
\end{equation*}
for large $n$ uniformly in $k$. Consequently,
\begin{equation}\label{e:Tn14}
 \bigg| \ln \prod_{k=1}^{k_{n}} T_{n,k}(f)  - \sum_{k=1}^{k_{n}}(T_{n,k}(f)-1) \bigg| \leq \sum_{k=1}^{k_{n}}(T_{n,k}(f)-1)^{2}
\end{equation}
for large $n$. Noting that by Lemma~\ref{l:lemma2} (i) and (ii) it holds that
$$ \sum_{k=1}^{k_{n}}(T_{n,k}(f)-1)^{2} \leq \sup_{1 \leq k \leq k_{n}}|T_{n,k}(f)-1| \cdot \sum_{k=1}^{k_{n}} |T_{n,k}(f)-1| \to 0 \qquad \textrm{as} \ n \to \infty,$$
relation (\ref{e:Tn14}) and the convergence relation in Lemma~\ref{l:lemma2} (ii) yield
$$ \lim_{n \to \infty} \ln \prod_{k=1}^{k_{n}} T_{n,k}(f) = \iint_{[0,1] \times \mathbb{E}} (e^{-f(t,x)}-1)\,\rmd t \nu(\rmd x).$$
Recalling (\ref{e:PsiNtilda}) and the definition of $T_{n,k}(f)$ we conclude that
$$ \lim_{n \to \infty} \ln \Psi_{\widehat{N}_{n}}(f) = \iint_{[0,1] \times \mathbb{E}} (e^{-f(t,x)}-1)\,\rmd t \nu(\rmd x).$$
From this and (\ref{e:diffppconv}) we immediately obtain
$$ \lim_{n \to \infty} \Psi_{N^{*}_{n}}(f) = \exp \bigg\{ - \iint_{[0,1] \times \mathbb{E}} (1- e^{-f(t,x)})\,\rmd t \nu(\rmd x) \bigg\},$$
from which, taking into account (\ref{e:LapPP}), we conclude that $\Psi_{N^{*}_{n}}(f) \to \Psi_{N^{*}}(f)$ as $n \to \infty$, where $N^{*}$ is a Poisson process with intensity measure $Leb \times \nu$. Finally, relation (\ref{e:ppLf}) implies $N_{n}^{*} \dto N^{*}$ as $n \to \infty$.
\end{proof}

In the following lemma we prove a version of Cline's theorem on the tail behaviour of weighted sums of heavy-tailed random variables. The original result holds in the i.i.d.~case (see Cline~\cite{Cl83}, Theorem 2.3; cf.~Resnick~\cite{Re87}, Lemma 4.24 and~Embrechts et al.~\cite{EmKlMi97}, Lemma A3.26), and here we extend it to the weak dependence case for nonnegative random variables and weights.

\begin{lem}\label{l:Cline}
Let $(Z_{i})_{i \geq 0}$ be a strictly stationary sequence of nonnegative regularly varying random variables with index $\alpha >0$ satisfying $(\ref{e:D'cond})$, and let $(c_{j})_{j \geq 0}$ be a sequence of nonnegative real numbers satisfying $(\ref{e:momcond})$. Then
\begin{equation}\label{e:Cline}
 \lim_{n \to \infty} \frac{\Pr \big( \sum_{j=0}^{\infty}c_{j}Z_{j} > a_{n} \big)}{\Pr(Z_{0}> a_{n})} = \sum_{j=0}^{\infty}c_{j}^{\alpha},
\end{equation}
where $(a_{n})_{n}$ is the sequence of real numbers defined in $(\ref{e:seqan})$.
\end{lem}
\begin{proof}
We first show that for two nonnegative random variables $U$ and $V$, $x>0$ and $0 < \epsilon < 1/2$ it holds that
\begin{equation}\label{e:Cline1new}
\Pr(U+V > x) \leq \Pr(U>(1-\epsilon)x) + \Pr(V >(1-\epsilon)x) + \Pr(U > \epsilon x, V > \epsilon x)
\end{equation}
and
\begin{equation}\label{e:Cline2new}
\Pr(U+V > x) \geq \Pr(U>(1+\epsilon)x, V \leq \epsilon x) + \Pr(V >(1+\epsilon)x, U \leq \epsilon x).
\end{equation}
Note that $\max\{U,V\} \leq (1-\epsilon)x$ and $\min\{U,V\} \leq \epsilon x$ imply $U+V \leq x$. This yields
$$\{U+V >x\} \subseteq \{U > (1-\epsilon)x\} \cup \{V > (1-\epsilon)x\} \cup \{U> \epsilon x, V > \epsilon x\},$$
 which immediately implies relation (\ref{e:Cline1new}). Since
$$ \{ U > (1+\epsilon)x, V \leq \epsilon x \} \cup  \{ V > (1+\epsilon)x, U \leq \epsilon x \} \subseteq \{ U + V > x \},$$
and the events on the left-hand side in this relation are disjoint, we see that (\ref{e:Cline2new}) holds.

We start with showing that (\ref{e:Cline}) holds in the case when only finitely many elements of the sequence $(c_{j})$ are nonzero. Without loss of generality assume $c_{0}, \ldots, c_{m}$ are nonzero for some $m \in \mathbb{N}$. Take an arbitrary $ \epsilon \in (0, 1/2)$. By (\ref{e:Cline1new}) we have
\begin{eqnarray}\label{e:Cline3new}
\nonumber \Pr \bigg( \sum_{j=0}^{m}c_{j}Z_{j} > a_{n} \bigg) &\leq& \Pr \bigg( \sum_{j=0}^{m-1}c_{j}Z_{j} > (1-\epsilon) a_{n} \bigg) + \Pr ( c_{m}Z_{m} > (1-\epsilon) a_{n} )\\[0.1em]
&& + \Pr \bigg( \sum_{j=0}^{m-1}c_{j}Z_{j} > \epsilon a_{n},\,c_{m}Z_{m} > \epsilon a_{n} \bigg).
\end{eqnarray}
Let $c= \max\{c_{j} : j=0,\ldots,m\}$. Since relations (\ref{e:seqan}) and (\ref{e:seqan1}) imply
$$ \frac{\Pr(c_{m}Z_{m}>(1-\epsilon)a_{n})}{\Pr(Z_{0}>a_{n})} = \frac{n \Pr(Z_{m}>(1-\epsilon) c_{m}^{-1} a_{n})}{n \Pr(Z_{0}>a_{n})} \to
 (1-\epsilon)^{-\alpha}c_{m}^{\alpha} \qquad \textrm{as} \ n \to \infty,$$
and by stationarity and relations (\ref{e:seqan}) and (\ref{e:D'cond}) (i.e.~(\ref{e:D'condrn})) it holds that
\begin{eqnarray}\label{e:Cline4new}
\nonumber \frac{\Pr \big( \sum_{j=0}^{m-1}c_{j}Z_{j} > \epsilon a_{n},\,c_{m}Z_{m} > \epsilon a_{n} \big)}{\Pr(Z_{0}>m)} &\leq& \sum_{j=0}^{m-1} \frac{\Pr (c_{j}Z_{j} > \epsilon a_{n}/m,\,c_{m}Z_{m} > \epsilon a_{n})}{\Pr(Z_{0}>m)}\\[0.3em]
& \hspace*{-30em} \leq& \hspace*{-15em} \sum_{j=0}^{m-1} \frac{n \Pr (Z_{j} > \epsilon a_{n}/(cm),\,Z_{m} > \epsilon a_{n}/(cm))}{n \Pr(Z_{0}>m)} \to 0 \qquad \textrm{as} \ n \to \infty,
\end{eqnarray}
from (\ref{e:Cline3new}) we obtain
\begin{equation*}
\limsup_{n \to \infty} \frac{\Pr \big( \sum_{j=0}^{m}c_{j}Z_{j} > a_{n} \big)}{\Pr(Z_{0}>a_{n})} \leq \limsup_{n \to \infty} \frac{\Pr \big( \sum_{j=0}^{m-1}c_{j}Z_{j} > (1-\epsilon) a_{n} \big)}{\Pr(Z_{0}>a_{n})} + (1-\epsilon)^{-\alpha}c_{m}^{\alpha}.
\end{equation*}
Repeating the same procedure for $\Pr \big( \sum_{j=0}^{m-1}c_{j}Z_{j} > (1-\epsilon) a_{n} \big)$ gives
\begin{equation*}
\limsup_{n \to \infty} \frac{\Pr \big( \sum_{j=0}^{m}c_{j}Z_{j} > a_{n} \big)}{\Pr(Z_{0}>a_{n})}
 \leq \limsup_{n \to \infty} \frac{\Pr \big( \sum_{j=0}^{m-2}c_{j}Z_{j} > (1-\epsilon)^{2} a_{n} \big)}{\Pr(Z_{0}>a_{n})} + (1-\epsilon)^{-2\alpha}c_{m-1}^{\alpha} + (1-\epsilon)^{-\alpha}c_{m}^{\alpha},
\end{equation*}
and therefore recursively we obtain
\begin{eqnarray}\label{e:Cline5new}
\nonumber \limsup_{n \to \infty} \frac{\Pr \big( \sum_{j=0}^{m}c_{j}Z_{j} > a_{n} \big)}{\Pr(Z_{0}>a_{n})}
 &\leq& \limsup_{n \to \infty} \frac{\Pr \big( c_{0}Z_{0} > (1-\epsilon)^{m} a_{n} \big)}{\Pr(Z_{0}>a_{n})} + \sum_{j=1}^{m} (1-\epsilon)^{-(m+1-j)\alpha}c_{j}^{\alpha}\\[0.2em]
  &=& (1-\epsilon)^{-m \alpha}c_{0}^{\alpha} + \sum_{j=1}^{m} (1-\epsilon)^{-(m+1-j)\alpha}c_{j}^{\alpha}.
\end{eqnarray}
On the other hand by (\ref{e:Cline2new}) we have
\begin{eqnarray}\label{e:Cline6new}
\nonumber \hspace*{-2em} \Pr \bigg( \sum_{j=0}^{m}c_{j}Z_{j} > a_{n} \bigg) &\geq& \Pr \bigg( \sum_{j=0}^{m-1}c_{j}Z_{j} > (1+\epsilon) a_{n},\,c_{m}Z_{m} \leq \epsilon a_{n} \bigg)\\[0.1em]
\nonumber &&  
+ \Pr \bigg(c_{m}Z_{m} > (1+\epsilon) a_{n},\,\sum_{j=0}^{m-1}c_{j}Z_{j} \leq \epsilon a_{n} \bigg)\\[0.1em]
\nonumber &
= & 
 \Pr \bigg( \sum_{j=0}^{m-1}c_{j}Z_{j} > (1+\epsilon) a_{n} \bigg) - \Pr \bigg( \sum_{j=0}^{m-1}c_{j}Z_{j} > (1+\epsilon) a_{n},\,c_{m}Z_{m} > \epsilon a_{n} \bigg)\\[0.1em]
&& 
+ \Pr (c_{m}Z_{m} > (1+\epsilon) a_{n}) -  \Pr \bigg(c_{m}Z_{m} > (1+\epsilon) a_{n},\,\sum_{j=0}^{m-1}c_{j}Z_{j} > \epsilon a_{n} \bigg).
\end{eqnarray}
As in (\ref{e:Cline4new}) it holds that, as $n \to \infty$,
\begin{eqnarray*}
\frac{\Pr \big( \sum_{j=0}^{m-1}c_{j}Z_{j} > (1+\epsilon) a_{n},\,c_{m}Z_{m} > \epsilon a_{n} \big)}{\Pr(Z_{0}>m)}
 \leq \frac{\Pr \big( \sum_{j=0}^{m-1}c_{j}Z_{j} > \epsilon a_{n},\,c_{m}Z_{m} > \epsilon a_{n} \big)}{\Pr(Z_{0}>m)} \to 0,
\end{eqnarray*}
and similarly
\begin{equation*}
\frac{ \Pr \big(c_{m}Z_{m} > (1+\epsilon) a_{n},\,\sum_{j=0}^{m-1}c_{j}Z_{j} > \epsilon a_{n} \big)}{\Pr(Z_{0}>m)} \to 0.
\end{equation*}
Therefore from (\ref{e:Cline6new}) we get
\begin{equation*}
\liminf_{n \to \infty} \frac{\Pr \big( \sum_{j=0}^{m}c_{j}Z_{j} > a_{n} \big)}{\Pr(Z_{0}>a_{n})} \geq \liminf_{n \to \infty} \frac{\Pr \big( \sum_{j=0}^{m-1}c_{j}Z_{j} > (1+\epsilon) a_{n} \big)}{\Pr(Z_{0}>a_{n})} + (1+\epsilon)^{-\alpha}c_{m}^{\alpha},
\end{equation*}
and again recursively we obtain
\begin{equation}\label{e:Cline7new}
 \liminf_{n \to \infty} \frac{\Pr \big( \sum_{j=0}^{m}c_{j}Z_{j} > a_{n} \big)}{\Pr(Z_{0}>a_{n})} \geq (1+\epsilon)^{-m \alpha}c_{0}^{\alpha} + \sum_{j=1}^{m} (1+\epsilon)^{-(m+1-j)\alpha}c_{j}^{\alpha}.
\end{equation}
Letting $\delta \to 0$ in (\ref{e:Cline5new}) and (\ref{e:Cline7new}) we see that
\begin{equation}\label{e:Cline8new}
\lim_{n \to \infty} \frac{\Pr \big( \sum_{j=0}^{m}c_{j}Z_{j} > a_{n} \big)}{\Pr(Z_{0}>a_{n})} = \sum_{j=0}^{m}c_{j}^{\alpha}.
\end{equation}
Having established relation (\ref{e:Cline}) for finitely many summands, the general case follows using the same arguments as in Lemma 4.24 in Resnick~\cite{Re87}.
\end{proof}

\begin{rem}\label{r:finitemomsum}
Note that condition $(\ref{e:momcond})$ implies $\sum_{j=0}^{\infty}|c_{j}|^{\beta} < \infty$ for all $\beta \geq \delta$, and hence in particular
$$  \sum_{j=0}^{\infty}|c_{j}|^{\alpha} < \infty.$$
\end{rem}

Theorem 2.4 (i) in Davis and Resnick~\cite{DaRe85} establishes convergence of point processes $N_{n}$ for stationary moving averages with i.i.d.~regularly varying innovations, while Theorem 2.1 in Niu~\cite{Ni97} generalizes this to the case of non-stationary moving averages with independent non-identically distributed innovations. Our main theorem below extends these results to the case of non-stationary linear processes with weakly dependent heavy-tailed innovations. In the proof we will follow the procedure from~\cite{DaRe85} with appropriate modifications due to non-stationarity of the innovations and the use of the dependence condition $D'$ instead of independence.

\begin{thm}\label{t:mainppc}
Let $(Z_{i})_{i \in \mathbb{Z}}$ be a strictly stationary sequence of random variables satisfying condition $(\ref{e:D'cond})$ and with regularly varying tail probabilities as specified by $(\ref{e:regvarpr1})$ and $(\ref{e:pq})$. Let $(\sigma_{i})_{i \in \mathbb{Z}}$ be a sequence of positive real numbers satisfying $(\ref{e:scaleparam1})$ and $(\ref{e:scaleparam2})$, and $(c_{j})_{j \geq 0}$ a sequence of real numbers satisfying $(\ref{e:momcond})$. Assume the mixing condition $(\ref{e:mixcon})$ holds. Then
\begin{equation}\label{e:ppconvNn}
 N_{n} \dto N := \sum_{i=1}^{\infty}\sum_{j=0}^{\infty}\delta_{(T_{i}, P_{i}c_{j})} \qquad \textrm{as} \ n \to \infty,
\end{equation}
in $M_p([0,1] \times \mathbb{E})$, where $\sum_{i=1}^{\infty}\delta_{(T_{i}, P_{i})}$ is a Poisson process with intensity measure $Leb \times \nu$, with $\nu$ as in Lemma~$\ref{l:lemma1}$.
\end{thm}
\begin{proof}
Let $m \geq 2$ be a fixed integer, and set $Y^{(k)}=(Y_{k}, Y_{k-1}, \ldots, Y_{k-m+1})$ for $k \in \mathbb{N}$, and
$$ I_{n} = \sum_{k=1}^{n}\delta_{(k/n, a_{n}^{-1}Y^{(k)})}, \qquad \ n \in \mathbb{N}.$$
Let $e_{j} \in \mathbb{R}^{m}$, for $j=1,\ldots,m$, be the basis element with $j$--th element equal to $1$ and the rest zero. Set
$$ I = \sum_{i=1}^{\infty}\sum_{j=1}^{m}\delta_{(T_{i}, P_{i}e_{j})}.$$
Let $S$ be the collection of all sets $B$ of the form
$$ B = (b_{0},d _{0}] \times (b_{1}, d_{1}] \times \ldots \times (b_{m}, d_{m}],$$
where $0 \leq b_{0} < d_{0} \leq 1$ and the $m$--dimensional rectangle $(b_{1}, d_{1}] \times \ldots \times (b_{m}, d_{m}]$ is bounded away from zero and $b_{j} < d_{j}$, $b_{j} \neq 0$, $d_{j} \neq 0$ for $j=1,\ldots,m$. Since $B \in S$, the set $(b_{1}, d_{1}] \times \ldots \times (b_{m}, d_{m}]$ either has empty intersection with all of the coordinate axes or intersects exactly one in an interval, i.e.~either
\begin{itemize}
\item[(C1)] $\quad (b_{1}, d_{1}] \times \ldots \times (b_{m}, d_{m}] \cap \{ye_{j} : y \in \mathbb{R}\} = \emptyset, \qquad \textrm{for} \ j=1,\ldots,m$
\end{itemize}
or
\begin{itemize}
\item[(C2)] $\quad (b_{1}, d_{1}] \times \ldots \times (b_{m}, d_{m}] \cap \{ye_{j} : y \in \mathbb{R}\} = \left\{ \begin{array}{cl}
  (b_{j'},d_{j'}], & \ \ \textrm{if} \ j=j',\\[0.5em]
 \emptyset, & \ \ \textrm{if} \ j \neq j'.
\end{array}\right.$
\end{itemize}
In the case (C2) observe that $b_{j} < 0 < d_{j}$ for $j \neq j'$ and $0 \notin (b_{j'}, d_{j'}]$. As noted in Davis and Resnick~\cite{DaRe85} it holds that
\begin{equation}\label{e:DaRe1}
 \Pr (I(\partial B) =0)=1 \qquad \textrm{for all} \ B \in S,
\end{equation}
and
\begin{equation}\label{e:DaRe2}
 \Pr(I(B)=0) = \left\{ \begin{array}{cl}
  1, & \ \ \textrm{if} \ B \in S \ \ \textrm{satisfies} \ (\textrm{C}1),\\[0.6em]
  e^{-Leb \times \nu ((b_{0},d_{0}] \times (b_{j'}, d_{j'}])}, & \ \ \textrm{if} \ B \in S \ \ \textrm{satisfies} \ (\textrm{C}2).
\end{array}\right.
\end{equation}
Now we show that
\begin{equation}\label{e:DaRe3}
 \lim_{n \to \infty} \mathrm{E} [I_{n}(B)] = \left\{ \begin{array}{cl}
  0, & \  \textrm{if} \ B \in S  \ \textrm{satisfies} \ (\textrm{C}1),\\[0.6em]
 Leb \times \nu ((b_{0},d_{0}] \times (b_{j'}, d_{j'}]), &  \ \textrm{if} \ B \in S  \ \textrm{satisfies} \ (\textrm{C}2).
\end{array}\right.
\end{equation}
Assume first $B$ satisfies (C1). Then $b_{j}<d_{j}<0$ or $0<b_{j}<d_{j}$ for all $j=1,\ldots,m$, and hence
\begin{eqnarray*}
  \mathrm{E}[I_{n}(B)] &=& \sum_{k/n \in (b_{0},d_{0}]} \Pr \bigg( \frac{Y_{k}}{a_{n}} \in (b_{1},d_{1}], \ldots, \frac{Y_{k-m+1}}{a_{n}} \in (b_{m},d_{m}] \bigg) \\[0.4em]
   &\leq &  \sum_{k/n \in (b_{0},d_{0}]} \Pr \bigg( \frac{|Y_{k}|}{a_{n}} > |b_{1}| \wedge |d_{1}|, \ldots, \frac{|Y_{k-m+1}|}{a_{n}} > |b_{m}| \wedge |d_{m}| \bigg),
\end{eqnarray*}
where $a \wedge b = \min \{a,b\}$.
Applying (\ref{e:scaleparam1}) and stationarity of the sequence $(Z_{i})$ we obtain
\begin{eqnarray}\label{e:DaRe4}
 \nonumber \mathrm{E}[I_{n}(B)]  &\leq &  \sum_{k=\floor{nb_{0}}+1}^{\floor{nd_{0}}} \Pr \bigg( \frac{|Z_{k}|}{a_{n}} > \frac{|b_{1}| \wedge |d_{1}|}{\beta_{2}}, \ldots, \frac{|Z_{k-m+1}|}{a_{n}} > \frac{|b_{m}| \wedge |d_{m}|}{\beta_{2}} \bigg).\\[0.4em]
 \nonumber  &\leq& (\floor{nd_{0}} - \floor{nb_{0}}) \Pr \bigg( \frac{|Z_{2}|}{a_{n}} > \frac{|b_{1}| \wedge |d_{1}|}{\beta_{2}}, \frac{|Z_{1}|}{a_{n}} > \frac{|b_{2}| \wedge |d_{2}|}{\beta_{2}} \bigg)\\[0.4em]
  &\leq& nd_{0} \Pr \bigg( \frac{|Z_{2}|}{a_{n}} > C_{1}, \frac{|Z_{1}|}{a_{n}} > C_{1} \bigg),
\end{eqnarray}
 where $C_{1} = \beta_{2}^{-1} \min \{ |b_{1}| \wedge |d_{1}|, |b_{2}| \wedge |d_{2}|\} >0$.
 For $B$ satisfying (C2) set
$$ E_{k,n} = \bigg\{  \frac{Y_{k-j'+1}}{a_{n}} \in (b_{j'},d_{j'}] \bigg\}$$
and
$$F_{k,n} = \bigg\{  \frac{Y_{k-j+1}}{a_{n}} \in (b_{j},d_{j}] \ \textrm{for all} \ j \in \{1,\ldots,m\} \setminus \{j'\} \bigg\},$$
and note that
\begin{equation}\label{e:DaRe5}
  \mathrm{E}[I_{n}(B)] = \sum_{k/n \in (b_{0},d_{0}]} \Pr(E_{k,n} \cap F_{k,n}) = \sum_{k/n \in (b_{0},d_{0}]} \Pr(E_{k,n}) - \sum_{k/n \in (b_{0},d_{0}]} \Pr(E_{k,n} \cap F_{k,n}^{c}).
\end{equation}
Using the same arguments as in the proof of Lemma~\ref{l:lemma1} we see that
\begin{eqnarray}\label{e:DaRe6}
   \nonumber \sum_{k/n \in (b_{0},d_{0}]} \Pr(E_{k,n}) &=& \sum_{k=1}^{n} \delta_{k/n}((b_{0},d_{0}]) \Pr \bigg( \frac{Y_{k-j'+1}}{a_{n}} \in (b_{j'},d_{j'}] \bigg)\\[0.4em]
  & \to& Leb \times \nu ((b_{0},d_{0}] \times (b_{j'},d_{j'}]) \qquad \textrm{as} \ n \to \infty.
\end{eqnarray}
Since
$$ F_{n,k}^{c}= \bigcup_{\scriptsize \begin{array}{l}
                          j=1  \\[0.1em]
                          j \neq j'
                        \end{array}}^{m} \bigg\{ \frac{Y_{k-j+1}}{a_{n}} \notin (b_{j},d_{j}] \bigg\} \subseteq \bigcup_{\scriptsize \begin{array}{l}
                          j=1  \\[0.1em]
                          j \neq j'
                        \end{array}}^{m} \bigg\{ \frac{|Y_{k-j+1}|}{a_{n}} \geq |b_{j}| \wedge |d_{j}| \bigg\}, $$
we have
\begin{eqnarray}\label{e:DaRe7}
 \nonumber  \sum_{k/n \in (b_{0},d_{0}]} \Pr(E_{k,n} \cap F_{k,n}^{c})
&\leq& \sum_{k=\floor{nb_{0}}+1}^{\floor{nd_{0}}} \sum_{\scriptsize \begin{array}{l}
                          j=1  \\[0.1em]
                          j \neq j'
                        \end{array}}^{m} \Pr \bigg(  \frac{Y_{k-j'+1}}{a_{n}} \in (b_{j'},d_{j'}], \frac{|Y_{k-j+1}|}{a_{n}} \geq |b_{j}| \wedge |d_{j}|\bigg)\\[0.4em]
\nonumber &\leq& \sum_{k=\floor{nb_{0}}+1}^{\floor{nd_{0}}} \sum_{\scriptsize \begin{array}{l}
                          j=1  \\[0.1em]
                          j \neq j'
                        \end{array}}^{m} \Pr \bigg(  \frac{|Z_{k-j'+1}|}{a_{n}} > \frac{|b_{j'}| \wedge |d_{j'}|}{\beta_{2}}, \frac{|Z_{k-j+1}|}{a_{n}} \geq \frac{|b_{j}| \wedge |d_{j}|}{\beta_{2}} \bigg)\\[0.4em]
&\leq& 2nd_{0} \sum_{j=2}^{m} \Pr \bigg(  \frac{|Z_{1}|}{a_{n}} > C_{2}, \frac{|Z_{j}|}{a_{n}} > C_{2} \bigg),
\end{eqnarray}
where $C_{2} = \min \{ \beta_{2}^{-1}(|b_{j}| \wedge |d_{j}|) : j=1,\ldots,m \}>0$. Observe that by (\ref{e:D'cond}) the last terms in (\ref{e:DaRe4}) and (\ref{e:DaRe7}) converge to $0$ as $n \to \infty$. This, together with (\ref{e:DaRe5}) and (\ref{e:DaRe6}) yields (\ref{e:DaRe3}).

Let
$$ \widetilde{I}_{n} = \sum_{k=1}^{n} \sum_{j=1}^{m} \delta_{(k/n, a_{n}^{-1}Y_{k} e_{j})}, \qquad \ n \in \mathbb{N}.$$
The next step is to show that, as $n \to \infty$,
\begin{equation}\label{e:DaRe8}
 \widetilde{I}_{n}(B) - I_{n}(B) \Pto 0 \qquad \textrm{for all} \ B \in S,
\end{equation}
where $\Pto$ denotes convergence in probability.
For $B \in S$ satisfying (C1), from the definition of $\widetilde{I}_{n}$, we have $\widetilde{I}_{n}(B)=0$. This, together with (\ref{e:DaRe3}) and Markov's inequality gives (\ref{e:DaRe8}). Take now $B \in S$ satisfying (C2). Using the fact that the points $Y_{k}e_{j}$ lie on the coordinate axes and Lemma~\ref{l:lemma1} we obtain
\begin{eqnarray}\label{e:DaRe9}
 \nonumber  \mathrm{E}[\widetilde{I}_{n}(B)] &=&  \sum_{k=1}^{n} \delta_{k/n}((b_{0},d_{0}]) \Pr \bigg( \frac{Y_{k}}{a_{n}} \in (b_{j'},d_{j'}] \bigg)\\[0.4em]
   &\to&  Leb \times \nu ((b_{0},d_{0}] \times (b_{j'},d_{j'}]) \qquad \textrm{as} \ n \to \infty.
\end{eqnarray}
For $n \geq j'$ write
$$ I_{n}(B)= \sum_{k=1}^{j'-1}\delta_{(k/n, a_{n}^{-1}Y^{(k)})}(B) + \sum_{k=j'}^{n}\delta_{(k/n, a_{n}^{-1}Y^{(k)})}(B) =: I^{*}_{1,n}+I^{*}_{2,n},$$
and note that
\begin{eqnarray*}
\nonumber \hspace*{-2.5em} \mathrm{E} I^{*}_{1,n} &\leq& \sum_{k=1}^{j'-1} \Pr \bigg( \frac{Y_{k}}{a_{n}} \in (b_{1},d_{1}], \ldots, \frac{Y_{k-m+1}}{a_{n}} \in (b_{m},d_{m}] \bigg)\\[0.3em]
\nonumber &\leq& \sum_{k=1}^{m} \Pr \bigg( \frac{Y_{k-j'+1}}{a_{n}} \in (b_{j'},d_{j'}] \bigg) \leq \sum_{k=1}^{m} \Pr \bigg( \frac{|Y_{k-j'+1}|}{a_{n}} > |b_{j'}| \wedge |d_{j'}| \bigg)\\[0.3em]
&\leq& \sum_{k=1}^{m} \Pr \bigg( \frac{|Z_{k-j'+1}|}{a_{n}} >  \frac{|b_{j'}| \wedge |d_{j'}|}{\beta_{2}} \bigg) = m \Pr \bigg( \frac{|Z_{1}|}{a_{n}} >  \frac{|b_{j'}| \wedge |d_{j'}|}{\beta_{2}} \bigg).
\end{eqnarray*}
Hence by (\ref{e:seqan1}) we obtain
\begin{equation}\label{e:DaRe10}
\limsup_{n \to \infty} \mathrm{E} I^{*}_{1,n} = 0.
\end{equation}
 Let $\gamma >0$ be arbitrary. Then $m/n \leq \gamma$ for large $n$. For these $n$ we have
\begin{eqnarray*}
 I^{*}_{2,n} & = & \sum_{k=j'}^{n}\delta_{(k/n, a_{n}^{-1}Y^{(k)})}(B) \leq  \sum_{k=j'}^{n}\delta_{(k/n, a_{n}^{-1}Y_{k-j'+1})}((b_{0},d_{0}] \times (b_{j'},d_{j'}])\\[0.4em]
 &=& \sum_{j=1}^{n-j'+1}\delta_{((j+j'-1)/n, a_{n}^{-1}Y_{j})}((b_{0},d_{0}] \times (b_{j'},d_{j'}])\\[0.4em]
  &\leq& \sum_{j=1}^{n}\delta_{(j/n, a_{n}^{-1}Y_{j})}((\widetilde{b}_{0},d_{0}] \times (b_{j'},d_{j'}]),
\end{eqnarray*}
where $\widetilde{b}_{0} = \max \{ b_{0}-\gamma,0\}$, and therefore
\begin{eqnarray*}
 I^{*}_{2,n} &\leq& \sum_{j=1}^{n}\delta_{(j/n, a_{n}^{-1}Y_{j})}((\widetilde{b}_{0},b_{0}] \times (b_{j'},d_{j'}]) + \sum_{j=1}^{n}\delta_{(j/n, a_{n}^{-1}Y_{j})}((b_{0},d_{0}] \times (b_{j'},d_{j'}])\\[0.4em]
 &=& \sum_{j=1}^{n}\delta_{(j/n, a_{n}^{-1}Y_{j})}((\widetilde{b}_{0},b_{0}] \times (b_{j'},d_{j'}]) + \widetilde{I}_{n}(B).
\end{eqnarray*}
Denoting
$ \widetilde{T}_{n,\gamma} = \sum_{j=1}^{n}\delta_{(j/n, a_{n}^{-1}Y_{j})}((\widetilde{b}_{0},b_{0}]\times (b_{j'},d_{j'}])$ we have
$$\widetilde{I}_{n}(B) + \widetilde{T}_{n,\gamma} -  I^{*}_{2,n} \geq 0,$$
and this yields
\begin{eqnarray}\label{e:DaRe11}
 \nonumber |\widetilde{I}_{n}(B) - I_{n}(B)| &=& |\widetilde{I}_{n}(B) - I^{*}_{1,n} - I^{*}_{2,n}| \leq  |(\widetilde{I}_{n}(B) + \widetilde{T}_{n,\gamma} - I^{*}_{2,n}) - \widetilde{T}_{n,\gamma}| +  I^{*}_{1,n}\\[0.4em]
 \nonumber  &\leq& |\widetilde{I}_{n}(B) + \widetilde{T}_{n,\gamma} - I^{*}_{2,n}| + \widetilde{T}_{n,\gamma} +  I^{*}_{1,n}
 = \widetilde{I}_{n}(B)
 + 2\widetilde{T}_{n,\gamma}- I^{*}_{2,n} +   I^{*}_{1,n}\\[0.4em]
 &=& \widetilde{I}_{n}(B) - I_{n}(B) + 2 \widetilde{T}_{n,\gamma} +  2I^{*}_{1,n}.
\end{eqnarray}
By Lemma~\ref{l:lemma1} it holds that
\begin{eqnarray*}
 \mathrm{E} \widetilde{T}_{n,\gamma} &=& \sum_{j=1}^{n}\delta_{j/n}((\widetilde{b}_{0},b_{0}]) \Pr \bigg(\frac{Y_{j}}{a_{n}} \in (b_{j'},d_{j'}] \bigg)\\[0.4em]
  &\to& Leb \times \nu ((\widetilde{b}_{0},b_{0}]\times (b_{j'},d_{j'}]) \leq \gamma \nu((b_{j'},d_{j'}]) \qquad \textrm{as} \ n \to \infty.
 \end{eqnarray*}
 This, together with relations (\ref{e:DaRe3}), (\ref{e:DaRe9}), (\ref{e:DaRe10}) and (\ref{e:DaRe11}) yields
  $$\limsup_{n \to \infty}\mathrm{E}|\widetilde{I}_{n}(B) - I_{n}(B)| \leq 2\gamma \nu((b_{j'},d_{j'}]).$$
  Letting $\gamma \to 0$ we see that
 \begin{equation*}
 \lim_{n \to \infty} \mathrm{E}|\widetilde{I}_{n}(B) - I_{n}(B)| = 0,
 \end{equation*}
 and therefore by an application of Markov's inequality we conclude that relation (\ref{e:DaRe8}) also holds in the case (C2).

 As noted in the proof of Theorem 2.2 in Davis and Resnick~\cite{DaRe85}, the mapping $\Psi \colon M_{p}([0,1] \times \EE ) \to M_{p}([0,1] \times ([-\infty, \infty]^{m} \setminus \{0\}))$ defined by
 $$ \Psi \bigg(\sum_{k=1}^{\infty} \delta_{(u_{k},v_{k})}\bigg) = \sum_{k=1}^{\infty} \sum_{j=1}^{m}\delta_{(u_{k},v_{k}e_{j})}$$
 is continuous, and hence by applying the continuous mapping theorem to the convergence result in Proposition~\ref{p:ppconvY} we obtain $\Psi(N^{*}_{n}) \dto \Psi(N^{*})$, i.e.~$\widetilde{I}_{n} \dto I$ as $n \to \infty$. From this and (\ref{e:DaRe8}), by Theorem 4.2 in Kallenberg~\cite{Ka83}, which can be applied since $S$ is a DC-semiring (cf.~Kallenberg~\cite{Ka83}, page 11), it follows that
 \begin{equation}\label{e:DaRe12}
 I_{n} \dto I \qquad \textrm{as} \ n \to \infty.
 \end{equation}
 As noted in the proof of Theorem 2.4 in Davis and Resnick~\cite{DaRe85}, the mapping
 $\Psi' \colon M_{p}([0,1] \times ([-\infty, \infty]^{m} \setminus \{0\})) \to M_{p}([0,1] \times \EE)$ defined by
 $$ \Psi' \bigg(\sum_{k=1}^{\infty} \delta_{(u_{k}, y_{k})}\bigg) = \sum_{k=1}^{\infty} \delta_{(u_{k}, \sum_{j=0}^{m-1}c_{j}y_{k}^{(j+1)})},$$
 where $y_{k}=(y_{k}^{(1)}, \ldots,y_{k}^{(m)}) \in [-\infty, \infty]^{m} \setminus \{0\}$, is continuous (cf.~Resnick~\cite{Re87}, page 235), and hence applying the continuous mapping theorem to the convergence in (\ref{e:DaRe12}) gives $\Psi'(I_{n}) \dto \Psi'(I)$ as $ n \to \infty$, that is
 \begin{equation}\label{e:DaRe13}
\sum_{k=1}^{n} \delta_{(k/n, a_{n}^{-1} \sum_{j=0}^{m-1}c_{j}Y_{k-j})} \dto \sum_{i=1}^{\infty} \sum_{j=0}^{m-1} \delta_{(T_{i}, P_{i}c_{j})}.
 \end{equation}
Observe that as $m \to \infty$
  \begin{equation}\label{e:DaRe14}
 \sum_{i=1}^{\infty} \sum_{j=0}^{m-1} \delta_{(T_{i}, P_{i}c_{j})} \to \sum_{i=1}^{\infty} \sum_{j=0}^{\infty} \delta_{(T_{i}, P_{i}c_{j})}
 \end{equation}
 pointwise in the vague metric $d_{v}$.
If we show that
 \begin{equation}\label{e:DaRe15}
 \lim_{m \to \infty} \limsup_{n \to \infty} \Pr \bigg[ \d_{v} \bigg( \sum_{k=1}^{n} \delta_{(k/n, a_{n}^{-1} \sum_{j=0}^{m-1}c_{j}Y_{k-j})}, \sum_{k=1}^{n} \delta_{(k/n, a_{n}^{-1}X_{k})} \bigg) > \epsilon \bigg]=0
 \end{equation}
 for any $\epsilon >0$, then from (\ref{e:DaRe13}) and (\ref{e:DaRe14}) by an application of a variant of Slutsky's theorem (see for example Resnick~\cite{Re87}, Lemma 4.25) it will follow that
 $$ N_{n} = \sum_{k=1}^{n} \delta_{(k/n, a_{n}^{-1}X_{k})} \dto N=\sum_{i=1}^{\infty} \sum_{j=0}^{\infty} \delta_{(T_{i}, P_{i}c_{j})} \qquad \textrm{as} \ n \to \infty.$$
 Because of the definition of $d_{v}$ relation (\ref{e:DaRe15}) will hold if for all $f \in C_{K}^{+}([0,1] \times \EE)$ and $\epsilon>0$
 \begin{equation}\label{e:DaRe16}
 \lim_{m \to \infty} \limsup_{n \to \infty} \Pr \bigg[ \bigg| \sum_{k=1}^{n} f \bigg( \frac{k}{n}, a_{n}^{-1} \sum_{j=0}^{m-1}c_{j}Y_{k-j} \bigg)- \sum_{k=1}^{n} f \bigg( \frac{k}{n}, a_{n}^{-1}X_{k} \bigg) \bigg| > \epsilon \bigg]=0
 \end{equation}
 (see Lemma~\ref{l:convprob} in Appendix A).

Since the support of $f$ is bounded away from origin it holds that
$f(t,x)=0$ for all $t \in [0,1]$ and $|x| \leq u$
 for some $u>0$.  Since $f$ is continuous on a compact set $[0,1] \times \{x \in \mathbb{E} : |x| \geq u/4 \}$, it is also uniformly continuous, and hence $\omega_{f,u}(\eta) \to 0$ as $\eta \to 0$, where
$$ \omega_{f,u}(\eta) = \sup \{ |f(t,x)-f(s,y)| : t,s \in [0,1],\,x,y \in \mathbb{E} \setminus (-u/4,u/4), d((t,x),(s,y)) \leq \eta \}$$
is the modulus of continuity of the function $f$ restricted to $[0,1] \times (\mathbb{E} \setminus (-u/4,u/4))$, and
$$d ((t,x), (s,y)) = \max \{ |t-s|, \rho(x,y) \}$$
is the metric on the direct product of metric spaces $[0,1]$ and $\mathbb{E}$
with $\rho$ defined in (\ref{e:metricrho}). Take $0 < \eta < \min\{u/4, 2/u\}$ and observe that on the event
$$ A(n,\eta) := \bigg\{ \frac{1}{a_{n}} \bigvee_{k=1}^{n} \bigg| \sum_{j=0}^{m-1}c_{j}Y_{k-j} - X_{k}\bigg| \leq \frac{\eta u^{2}}{8} \bigg\},$$
 $a_{n}^{-1} |\sum_{j=0}^{m-1}c_{j}Y_{k-j}| \leq u/2$ implies $f(k/n, a_{n}^{-1}\sum_{j=0}^{m-1}c_{j}Y_{k-j}) = f(k/n, a_{n}^{-1}X_{k})=0$, and
 $a_{n}^{-1} |\sum_{j=0}^{m-1}c_{j}Y_{k-j}| > u/2$ implies
 $$ \bigg|f \bigg( \frac{k}{n}, \frac{1}{a_{n}}\sum_{j=0}^{m-1}c_{j}Y_{k-j} \bigg) - f \bigg( \frac{k}{n}, \frac{X_{k}}{a_{n}} \bigg) \bigg| \leq \omega_{f,u}(\eta),$$
 since
 $$ \frac{|X_{k}|}{a_{n}} \geq \bigg| \frac{1}{a_{n}} \sum_{j=0}^{m-1}c_{j}Y_{k-j}\bigg| - \bigg|  \frac{1}{a_{n}} \sum_{j=0}^{m-1}c_{j}Y_{k-j} - \frac{X_{k}}{a_{n}} \bigg| \geq \frac{u}{2} - \frac{\eta u^{2}}{8} \geq \frac{u}{4}$$
 and
 $$ \rho \bigg(  \frac{1}{a_{n}} \sum_{j=0}^{m-1}c_{j}Y_{k-j}, \frac{X_{k}}{a_{n}} \bigg) \leq \frac{| a_{n}^{-1} \sum_{j=0}^{m-1}c_{j}Y_{k-j} - a_{n}^{-1}X_{k}|}{| a_{n}^{-1} \sum_{j=0}^{m-1}c_{j}Y_{k-j}| \cdot |a_{n}^{-1} X_{k}|} \leq \frac{\eta u^{2}/8}{(u/2) \cdot (u/4)}=\eta. $$
Therefore
\begin{eqnarray*}
  \Pr \bigg[ \bigg\{ \bigg| \sum_{k=1}^{n} f \bigg( \frac{k}{n}, a_{n}^{-1} \sum_{j=0}^{m-1}c_{j}Y_{k-j} \bigg)- \sum_{k=1}^{n} f \bigg( \frac{k}{n}, a_{n}^{-1}X_{k} \bigg) \bigg| > \epsilon \bigg\} \cap A(n, \eta) \bigg] &&\\[0.3em]
   & \hspace*{-54em} \leq & \hspace*{-27em} \Pr \bigg[  \omega_{f,u}(\eta) \sum_{k=1}^{n} \delta_{ (k/n, a_{n}^{-1} \sum_{j=0}^{m-1}c_{j}Y_{k-j}) } ([0,1] \times \{y : |y| > u/2\}) > \epsilon \bigg],
\end{eqnarray*}
and by (\ref{e:DaRe13}) this converges as $n \to \infty$ to
$$ \Pr \bigg[  \omega_{f,u}(\eta) \sum_{i=1}^{\infty} \sum_{j=0}^{m-1} \delta_{ (T_{i}, P_{i}c_{j}) } ([0,1] \times \{y : |y| > u/2\}) > \epsilon \bigg],$$
and as $m \to \infty$ by (\ref{e:DaRe14}) this converges to
$$ \Pr \bigg[  \omega_{f,u}(\eta) \sum_{i=1}^{\infty} \sum_{j=0}^{\infty} \delta_{ (T_{i}, P_{i}c_{j}) } ([0,1] \times \{y : |y| > u/2\}) > \epsilon \bigg].$$
Since $\sum_{i=1}^{\infty} \sum_{j=0}^{\infty} \delta_{ (T_{i}, P_{i}c_{j}) } ([0,1] \times \{y : |y| > u/2\}) < \infty$ a.s., the preceding probability tends to $0$ as $\eta \to 0$, i.e.~we have
\begin{equation}\label{e:DaRe17}
 \lim_{\eta \to 0} \lim_{m \to \infty} \limsup_{n \to \infty} \Pr \bigg[ \bigg\{ \bigg| \sum_{k=1}^{n} f \bigg( \frac{k}{n}, a_{n}^{-1} \sum_{j=0}^{m-1}c_{j}Y_{k-j} \bigg)- \sum_{k=1}^{n} f \bigg( \frac{k}{n}, a_{n}^{-1}X_{k} \bigg) \bigg| > \epsilon \bigg\} \cap A(n, \eta) \bigg]=0.
\end{equation}
Note that by (\ref{e:scaleparam1}) and stationarity of the sequence $(Z_{i})$ it holds that
\begin{eqnarray*}
\Pr(A(n,\eta)^{c}) &\leq& \sum_{k=1}^{n} \Pr \bigg( \frac{1}{a_{n}} \bigg| \sum_{j=0}^{m-1}c_{j}Y_{k-j} - X_{k}\bigg| > \frac{\eta u^{2}}{8} \bigg)
 = \sum_{k=1}^{n} \Pr \bigg( \frac{1}{a_{n}} \bigg| \sum_{j=m}^{\infty}c_{j}Y_{k-j}\bigg| > \frac{\eta u^{2}}{8} \bigg)\\[0.3em]
 &\leq & \sum_{k=1}^{n} \Pr \bigg( \frac{1}{a_{n}} \sum_{j=m}^{\infty}|c_{j}| \cdot |Z_{k-j}|  > \frac{\eta u^{2}}{8 \beta_{2}} \bigg)
 = n \Pr \bigg( \frac{1}{a_{n}} \sum_{j=m}^{\infty}|c_{j}| \cdot |Z_{-j}|  > \frac{\eta u^{2}}{8 \beta_{2}} \bigg).
\end{eqnarray*}
Applying Lemma~\ref{l:Cline} and taking into account (\ref{eq:niz}) we obtain
\begin{eqnarray*}
\limsup_{n \to \infty} \Pr(A(n,\eta)^{c}) &\leq& \limsup_{n \to \infty} n \Pr \bigg( \frac{1}{a_{n}} \sum_{j=m}^{\infty}|c_{j}| \cdot |Z_{-j}|  > \frac{\eta u^{2}}{8 \beta_{2}} \bigg)
= \Big( \frac{\eta u^{2}}{8 \beta_{2}} \Big)^{-\alpha} \sum_{j=m}^{\infty}|c_{j}|^{\alpha},
\end{eqnarray*}
and this by (\ref{e:momcond}) tends to $0$ as $m \to \infty$. Hence
\begin{equation*}
  \lim_{m \to \infty} \limsup_{n \to \infty} \Pr \bigg[ \bigg\{ \bigg| \sum_{k=1}^{n} f \bigg( \frac{k}{n}, a_{n}^{-1} \sum_{j=0}^{m-1}c_{j}Y_{k-j} \bigg)- \sum_{k=1}^{n} f \bigg( \frac{k}{n}, a_{n}^{-1}X_{k} \bigg) \bigg| > \epsilon \bigg\} \cap A(n, \eta)^{c} \bigg]=0,
\end{equation*}
and from this and (\ref{e:DaRe17}) we obtain (\ref{e:DaRe16}), proving the result.
\end{proof}

\section{Partial maxima processes}\label{s:partialmp}

At the beginning of this section we first recall the notion of Skorokhod $M_{1}$ topology.  For $x \in D([0,1],
\mathbb{R})$ the completed graph of $x$ is the set
\[
  \Gamma_{x}
  = \{ (t,z) \in [0,1] \times \mathbb{R} : z = \lambda x(t-) + (1-\lambda) x(t) \ \textrm{for some} \ \lambda \in [0,1] \},
\]
where $x(t-)$ is the left limit of $x$ at $t$.
We define an
order on the graph $\Gamma_{x}$ by saying that $(t_{1},z_{1}) \le
(t_{2},z_{2})$ if either (i) $t_{1} < t_{2}$ or (ii) $t_{1} = t_{2}$
and $|x_{j}(t_{1}-) - z_{1}| \le |x_{j}(t_{2}-) - z_{2}|$. A parametric representation
of the graph $\Gamma_{x}$ is a continuous nondecreasing function $(r,u)$
mapping $[0,1]$ onto $\Gamma_{x}$, with $r$ being the
time component and $u$ being the spatial component. Let $\Pi(x)$ denote the set of
parametric representations of the graph $\Gamma_{x}$. For $x_{1},x_{2}
\in D([0,1], \mathbb{R})$ define
\[
  d_{M_{1}}(x_{1},x_{2})
  = \inf \{ \|r_{1}-r_{2}\|_{[0,1]} \vee \|u_{1}-u_{2}\|_{[0,1]} : (r_{i},u_{i}) \in \Pi(x_{i}), i=1,2 \}.
\]
where $a \vee b= \max\{a,b\}$ for $a,b \in \mathbb{R}$, and $\|x\|_{[0,1]} = \sup \{ |x(t)| : t \in [0,1] \}$. The distance function $d_{M_{1}}$ is a metric on $D([0,1], \mathbb{R})$, and the induced topology is called the (standard or strong) Skorokhod $M_{1}$ topology.
For more discussion of the $M_{1}$ topology we refer to sections 12.3--12.5 in Whitt~\cite{Whitt02}.

Let $(Z_{i})_{i \in \mathbb{Z}}$ be a strictly stationary sequence of random variables satisfying the local dependence condition $(\ref{e:D'cond})$ and with regularly varying tail probabilities with index $\alpha > 0$ as specified by $(\ref{e:regvarpr1})$ and $(\ref{e:pq})$. Let $(\sigma_{i})_{i \in \mathbb{Z}}$ be a sequence of positive real numbers satisfying $(\ref{e:scaleparam1})$ and $(\ref{e:scaleparam2})$, and $(c_{j})_{j \geq 0}$ a sequence of real numbers satisfying $(\ref{e:momcond})$. Set
$$c_{+} = \max \{ c_{j} \vee 0 : j \geq 0\} \qquad \textrm{and} \qquad c_{-} = \max \{ -c_{j} \vee 0 : j \geq 0 \}.$$
Let $(X_{n})_{n}$ be a sequence of linear processes defined by
\begin{equation*}\label{e:linproc1sn}
X_{n} = \sum_{j=0}^{\infty}c_{j}\sigma_{n-j}Z_{n-j}, \qquad n \in \mathbb{N}.
\end{equation*}
For $n \in \mathbb{N}$ let
$$ W_{n}(t) =   \bigvee_{k=1}^{\lfloor nt \rfloor} \frac{X_{k}}{a_{n}},
   \qquad t \in [0,1],$$
with the convention $W_{n}(t) = X_{1}/a_{n}$ for $t \in [0, 1/n)$,
where $a_{n}$ as in (\ref{e:seqan}), i.e. $a_{n}$ is the $(1-n^{-1})$--quantile of the distribution function of $|Z_{1}|$,

In order to establish the functional convergence of $W_{n}$ we first represent this process as the image, under an appropriate maximum functional, of the
time-space point processes $N_n$ defined in (\ref{e:ppNn}). Then, using certain continuity properties of this
functional, by the continuous mapping theorem we
transfer the weak convergence of $N_n$ in (\ref{e:ppconvNn}) to
weak convergence of $W_n$.

Define the maximum functional
$ \Phi \colon M_{p}([0,1] \times \EE) \to D([0,1], \mathbb{R})$
by
$$ \Phi \Big( \sum_{i}\delta_{(t_{i}, x_{i})} \Big) (t)
  =   \bigvee_{t_{i} \leq t}x_{i} \vee 0, \qquad t \in [0,1],$$
  where the space $M_p([0,1] \times \EE)$ of Radon point
measures on $[0,1] \times \EE$ is equipped with the vague
topology, and we set for convenience $\vee \emptyset =0$. Let
\begin{equation*}
\Lambda  = \{ \eta \in M_{p}([0,1] \times \EE) :
   \eta ( \{0,1 \} \times \EE) = 0 = \eta ([0,1] \times \{ \pm \infty \}) \}.
\end{equation*}
Using the same arguments from the proof of Lemma 3.2 in Krizmani\'{c} \cite{Kr20} we obtain that the functional $\Phi$ is continuous on the set $\Lambda$ when $D([0,1], \mathbb{R})$ is endowed with the $M_{1}$ topology.
Note that the Poisson point process $N^{*}=\sum_{i=1}^{\infty}\delta_{(T_{i}, P_{i})}$ from Proposition~\ref{p:ppconvY} is almost surely contained in the set $\Lambda$ (cf.~Resnick \cite{Resnick07}, page 221), and therefore the limiting point process $N= \sum_{i=1}^{\infty}\sum_{j=0}^{\infty}\delta_{(T_{i}, P_{i}c_{j})}$ from Theorem~\ref{t:mainppc} also almost surely belongs to the set $\Lambda$.

In the theorem below we establish functional convergence of the process $W_{n}$, with the limit $W$ being an extremal process.
Recall that the distribution of a nonnegative extremal process $W$ is characterized by its exponent measure $\nu''$ in the following way:
$$ \PP (W(t) \leq x ) = e^{-t \nu''(x,\infty)}$$
for $t>0$ and $x>0$, where $\nu''$ is a measure on $(0,\infty)$ satisfying
$ \nu'' (\delta, \infty) < \infty$
for any $\delta >0$ (see Resnick~\cite{Resnick07}, page 161).

\begin{thm}\label{t:fconvmax}
Let $(Z_{i})_{i \in \mathbb{Z}}$ be a strictly stationary sequence of regularly varying random variables satisfying $(\ref{e:regvarpr1})$ and $(\ref{e:pq})$ with $\alpha >0$, such that the local dependence condition $(\ref{e:D'cond})$ holds. Let $(\sigma_{i})_{i \in \mathbb{Z}}$ be a sequence of positive real numbers satisfying $(\ref{e:scaleparam1})$ and $(\ref{e:scaleparam2})$, and $(c_{j})_{j \geq 0}$ a sequence of real numbers satisfying $(\ref{e:momcond})$. Assume the mixing condition $(\ref{e:mixcon})$ holds and $pc_{+} + qc_{-}>0$, with $p$ and $q$ as in $(\ref{e:pq})$. Then
$$W_{n} \dto W \qquad \textrm{as} \ n \to \infty,$$
in $D([0,1], \mathbb{R})$ endowed with the $M_{1}$ topology, where $W$
is an extremal process with exponent measure
\begin{equation*}
\nu^{*}(\rmd x)  = \sigma^{\alpha} (p c_{+}^{\alpha} + q c_{-}^{\alpha}) \alpha x^{-\alpha-1} 1_{(0,\infty)}(x)\,\rmd x.
\end{equation*}
\end{thm}
\begin{proof}
Since the maximum functional $\Phi$ is continuous on the set $\Lambda$, and this set almost surely contains the point process $N$ from Theorem~\ref{t:mainppc}, the continuous mapping theorem applied to the convergence in (\ref{e:ppconvNn}) yields
$\Phi^{(u)}(N_{n}) \dto \Phi^{(u)}(N)$ in $D([0,1], \mathbb{R})$ under the $M_{1}$ topology, that is
\begin{equation}\label{e:mainconvC2max}
        W_{n}^{*}(\,\cdot\,) := \bigvee_{i = 1}^{\lfloor n \, \cdot \, \rfloor} \frac{X_{i}}{a_{n}} \vee 0 \ \dto \ \bigvee_{T_{i} \le \cdot} \bigvee_{j \geq 0}P_{i}c_{j} \vee 0 \qquad \textrm{as} \ n \to \infty.
\end{equation}
By Proposition 5.2 in Resnick~\cite{Resnick07} the process
$$N^{+} := \sum_{i} \delta_{(T_{i}, P_{i} 1_{\{ P_{i} > 0\}}c_{+})}$$
 is a Poisson process with intensity measure
$ Leb \times \nu^{+}$, where
$$ \nu^{+}(\rmd x)  = \sigma^{\alpha} p c_{+}^{\alpha} \alpha x^{-\alpha-1} 1_{(0,\infty)}(x)\,\rmd x$$
and 
$$N^{-} := \sum_{i} \delta_{(T_{i}, - P_{i} 1_{\{ P_{i} < 0\}}c_{-})}$$
 is a Poisson process with intensity measure
$ Leb \times \nu^{-}$, where
$$ \nu^{-}(\rmd x)  = \sigma^{\alpha} q c_{-}^{\alpha} \alpha x^{-\alpha-1} 1_{(0,\infty)}(x)\,\rmd x$$
(for details see Lemma~\ref{l:ItoPRM} in Appendix A).
Hence
$$ W^{+}(\,\cdot\,) := \bigvee_{T_{i} \le \cdot} P_{i} 1_{\{P_{i}>0\}}c_{+} \qquad \textrm{and} \qquad
 W^{-}(\,\cdot\,) := \bigvee_{T_{i} \le \cdot} \Big( -P_{i} 1_{\{P_{i}<0\}}c_{-} \Big)
$$
are extremal processes with exponent measures $\nu^{+}$ and $\nu^{-}$, respectively (see Resnick~\cite{Resnick07}, page 161).

Note that the processes $N^{+}$ and $N^{-}$ are independent, since the first comes from those points $(T_{i}, P_{i})$ of the Poisson process $\sum_{i=1}^{\infty}\delta_{(T_{i},P_{i})}$ with $P_{i}>0$, and the second from points with $P_{i}<0$. Hence $N^{+}+N^{-}$ is a Poisson process with intensity measure $Leb \times (\nu^{+}+\nu^{-})$ (cf.~the proof of Proposition 4.28 in Resnick~\cite{Re87}), and
$$ W(\,\cdot\,) := \bigvee_{T_{i} \le \cdot} \Big[P_{i}1_{\{ P_{i}> 0\}}c_{+} \vee \Big(-P_{i}1_{\{ P_{i}< 0\}}c_{-} \Big) \Big]$$
is an extremal process with exponent measure $\nu^{+}+\nu^{-}=\nu^{*}$ (see Resnick~\cite{Resnick07}, page 161).
Since for nonnegative real numbers $B, b_{1}, b_{2}, \ldots$ it holds that
$$ \bigvee_{j \geq 1} B b_{j} = B \bigvee_{j \geq 1} b_{j},$$ we have
\begin{eqnarray*}
 W(\,\cdot\,) &=& \bigvee_{T_{i} \le \cdot} \bigg[ \bigg( P_{i}1_{\{ P_{i}> 0\}} \bigvee_{j \geq 0}c_{j} 1_{\{c_{j}>0\}} \bigg) \vee \bigg(-P_{i}1_{\{ P_{i}< 0\}} \bigvee_{j \geq 0} (-c_{j} 1_{\{c_{j}<0\}}) \bigg)\bigg]\\[0.4em]
  &=& \bigvee_{T_{i} \le \cdot} \bigg[ \bigg( \bigvee_{j \geq 0} P_{i}1_{\{ P_{i}> 0\}} c_{j} 1_{\{c_{j}>0\}} \bigg) \vee \bigg( \bigvee_{j \geq 0} P_{i}1_{\{ P_{i}< 0\}} c_{j} 1_{\{c_{j}<0\}} \bigg)\bigg]\\[0.4em]
  &=& \bigvee_{T_{i} \le \cdot} \bigvee_{j \geq 0}P_{i}c_{j} \vee 0,
 \end{eqnarray*}
 which shows that the extremal process $W$ is in fact the limiting process in (\ref{e:mainconvC2max}).

If we show that
$$\lim_{n \to \infty} \Pr (d_{M_{1}}(W_{n}, W_{n}^{*}) > \epsilon)=0$$
for any $\epsilon >0$, from (\ref{e:mainconvC2max}) by Slutsky's theorem (see Theorem 3.4 in Resnick~\cite{Resnick07}) it will follow that
$ W_{n} \dto W$ as $n \to \infty$,
in $D([0,1], \mathbb{R})$ with the $M_{1}$ topology. Using the fact that the
 metric $d_{M_{1}}$ on $D([0,1], \mathbb{R})$ is bounded above by the uniform metric and the fact that
 $$\bigg| \bigvee_{i=1}^{\lfloor nt \rfloor}\frac{X_{i}}{a_{n}} - \bigvee_{i=1}^{ \lfloor nt \rfloor }\frac{X_{i}}{a_{n}} \vee 0 \bigg| \leq \frac{|X_{1}|}{a_{n}} 1_{\{ X_{1} < 0 \}},$$
 we obtain
\begin{eqnarray}\label{e:slutsky}
  \nonumber \Pr (d_{M_{1}}(W_{n}, W_{n}^{*}) > \epsilon) & \leq &
   \Pr \bigg(
     \sup_{t \in [0,1]} |W_{n}(t) - W_{n}^{*}(t)| >  \epsilon \bigg)\\[0.4em]
   \nonumber &=& \Pr \bigg( \sup_{t \in [0,1]} \bigg| \bigvee_{i=1}^{\lfloor nt \rfloor}\frac{X_{i}}{a_{n}} - \bigvee_{i=1}^{ \lfloor nt \rfloor }\frac{X_{i}}{a_{n}} \vee 0 \bigg| > \epsilon  \bigg)\\[0.4em]
   & \leq & \Pr \bigg( \frac{|X_{1}|}{a_{n}} 1_{\{ X_{1} < 0 \}} > \epsilon \bigg)
 \end{eqnarray}
  Since
 $$ \frac{|X_{1}|}{a_{n}} 1_{\{X_{1} < 0\}} \to 0$$
 almost surely as $n \to \infty$, it follows that
 \begin{equation*}
   \PP \bigg( \frac{|X_{1}|}{a_{n}} 1_{\{ X_{1} < 0 \}} > \epsilon
       \bigg) \to 0 \qquad \textrm{as} \ n \to \infty,
 \end{equation*}
and hence from (\ref{e:slutsky}) we get
$$\lim_{n \to \infty} \Pr (d_{M_{1}}(W_{n}, W_{n}^{*}) > \epsilon)=0,$$
and this concludes the proof.
\end{proof}

\appendix
\section{}

Here we first show that the mixing condition $\mathcal{A}'$, that is relation (\ref{e:mixcon}) is implied by the strong mixing property.
Recall the notion of strong mixing. On a given probability space $(\Omega, \mathcal{F}, \Pr)$, for two $\sigma$--fields $\mathcal{D}_{1}, \mathcal{D}_{2} \subseteq \mathcal{F}$ define
$$ \alpha (\mathcal{D}_{1}, \mathcal{D}_{2}) = \sup \{ |\Pr(D_{1} \cap D_{2}) - \Pr(D_{1}) \Pr(D_{2})| :
   D_{1} \in \mathcal{D}_{1}, D_{2} \in \mathcal{D}_{2} \}.$$
We say that a sequence
of random variables $(\xi_{n})_{n \in \mathbb{Z}}$ satisfies the
strong mixing or $\alpha$-mixing condition if
$$ \alpha_{n} = \sup_{j \in \mathbb{Z}} \alpha (\mathcal{F}_{-\infty}^{j}, \mathcal{F}_{j+n}^{\infty})  \to 0 \qquad \textrm{as} \ n \to \infty,$$
 where
$\mathcal{F}_{-\infty}^{k}= \sigma (\{ \ldots , \xi_{k-1}, \xi_{k} \})$
 and
 $\mathcal{F}_{k}^{\infty}= \sigma (\{\xi_{k}, \xi_{k+1},\ldots  \})$ for $k \in \mathbb{Z}$. The numbers $\alpha_{n}$ are called the strong mixing coefficients. Note that when the sequence $(\xi_{n})$ is strictly stationary, one has simply $\alpha_{n} = \alpha (\mathcal{F}_{-\infty}^{0}, \mathcal{F}_{n}^{\infty})$.

 Let $(Z_{i})_{i \in \mathbb{Z}}$ be a strictly stationary sequence of random variables with regularly varying tail probabilities as specified by $(\ref{e:regvarpr1})$ and $(\ref{e:pq})$, and let $(\sigma_{i})_{i \in \mathbb{Z}}$ be a sequence of positive real numbers satisfying $(\ref{e:scaleparam1})$. Set $Y_{i}= \sigma_{i}Z_{i}$ for $i \in \mathbb{Z}$. Let $\mathbb{E}= [-\infty, \infty] \setminus \{0\}$.

\begin{lemma}\label{l:A'strongmix}
If $(Z_{i})_{i}$ is strongly mixing, then condition $\mathcal{A}'$ holds, i.e.~there exists a sequence of positive integers $(r_{n})_{n}$ such that $r_{n} \to \infty$ and $r_{n}/n \to 0$ as $n \to \infty$, such that for all $f \in C_{K}^{+}([0,1] \times \mathbb{E})$, as $n \to \infty$,
$$ \E \biggl[ \exp \biggl\{ - \sum_{i=1}^{n} f \biggl(\frac{i}{n}, \frac{Y_{i}}{a_{n}}
 \biggr) \biggr\} \biggr]
 - \prod_{k=1}^{k_{n}} \E \biggl[ \exp \biggl\{ - \sum_{i=1}^{r_{n}} f \biggl(\frac{kr_{n}}{n}, \frac{Y_{i+(k-1)r_{n}}}{a_{n}} \biggr) \biggr\} \biggr] \to 0,$$
where $k_{n}= \lfloor n/r_{n} \rfloor$ and $(a_{n})_{n}$ is a sequence of positive real numbers tending to infinity such that
$(\ref{e:seqan})$ holds.
\end{lemma}
\begin{proof}
Since the sequence $(\sigma_{i})_{i}$ is deterministic, $\sigma(Y_{i})=\sigma(Z_{i})$ and hence the process $(Y_{i})_{i}$ is also strongly mixing with the same strong mixing coefficients as the process $(Z_{i})_{i}$.

Let $(l_{n})_{n \in \mathbb{N}}$ be an arbitrary sequence of positive
integers such that $l_{n} \to \infty$ as $n \to \infty$ and
 $l_{n} = o(n^{1/8})$, where $b_{n}=o(d_{n})$ means $b_{n}/d_{n} \to 0$ as $n \to \infty$.
 Put
 \begin{equation}\label{e:rn1}
 r_{n} = \lfloor \max \{ n \sqrt{ \alpha_{l_{n}+1}},\,n^{2/3} \} \rfloor
    +1, \qquad n \in \mathbb{N},
 \end{equation}
 where $(\alpha_{n})_{n}$ is the sequence of strong mixing coefficients of both sequences $(Z_{i})_{i}$ and $(Y_{i})_{i}$.
 Then $r_{n} \to \infty$ as $n \to \infty$.
 Since the sequence $(Z_{i})$ is strongly mixing,
  $\alpha_{l_{n}+1} \to 0$ as $n \to \infty$, and therefore
   $r_{n}/n \to 0$ as $n \to \infty$.
   Hence it follows that  $k_{n}
  \to \infty$ and
\begin{equation}\label{e:strmix1}
    k_{n} \alpha_{l_{n}+1} \to 0 \qquad \textrm{and} \qquad
    \frac{k_{n}l_{n}}{n} \to 0.
\end{equation}
Fix $f \in C_{K}^{+}([0,1] \times \mathbb{E})$. We have to show that
$V(n) \to 0$ as $n \to \infty$, where
$$ V(n) =  \bigg| \mathrm{E}  \bigg[ \exp \bigg\{ - \sum_{i=1}^{n} f \bigg(\frac{i}{n}, \frac{Y_{i}}{a_{n}}
 \bigg) \bigg\} \bigg]
 - \prod_{k=1}^{k_{n}} \mathrm{E} \bigg[ \exp \bigg\{ - \sum_{i=1}^{r_{n}} f \bigg(\frac{kr_{n}}{n}, \frac{Y_{i+(k-1)r_{n}}}{a_{n}} \bigg)
 \bigg\} \bigg] \bigg|.$$
 We have
\begin{eqnarray}\label{e:I}
  \nonumber V(n) & \leq & \bigg| \mathrm{E} \bigg[ \exp \bigg\{ - \sum_{i=1}^{n} f \bigg(\frac{i}{n}, \frac{Y_{i}}{a_{n}}
 \bigg) \bigg\} \bigg]
 - \mathrm{E} \bigg[ \exp \bigg\{ - \sum_{i=1}^{k_{n}r_{n}} f \bigg(\frac{i}{n}, \frac{Y_{i}}{a_{n}} \bigg)
 \bigg\} \bigg] \bigg| \\[0.5em]
   \nonumber & & \hspace*{-3.7em} + \ \bigg| \mathrm{E} \bigg[  \exp \bigg\{ - \sum_{i=1}^{k_{n}r_{n}} f \bigg(\frac{i}{n}, \frac{Y_{i}}{a_{n}}
 \bigg) \bigg\} \bigg]
 - \mathrm{E} \bigg[ \exp \bigg\{ - \sum_{k=1}^{k_{n}} \sum_{i=(k-1)r_{n}+1}^{kr_{n}-l_{n}}
   f \bigg(\frac{i}{n}, \frac{Y_{i}}{a_{n}} \bigg)
 \bigg\} \bigg] \bigg| \\[0.5em]
 \nonumber & &  \hspace*{-3.7em} + \ \bigg| \mathrm{E} \bigg[ \exp \bigg\{ - \sum_{k=1}^{k_{n}} \sum_{i=(k-1)r_{n}+1}^{kr_{n}-l_{n}}
    f \bigg(\frac{i}{n}, \frac{Y_{i}}{a_{n}} \bigg) \bigg\} \bigg]
 - \prod_{k=1}^{k_{n}} \mathrm{E} \bigg[ \exp \bigg\{ - \sum_{i=1}^{r_{n}-l_{n}} f \bigg(\frac{kr_{n}}{n}, \frac{Y_{i+(k-1)r_{n}}}{a_{n}} \bigg)
 \bigg\} \bigg] \bigg| \\[0.5em]
   \nonumber & & \hspace*{-3.7em} + \ \bigg| \prod_{k=1}^{k_{n}} \mathrm{E} \bigg[ \exp \bigg\{ - \sum_{i=1}^{r_{n}-l_{n}} f \bigg(\frac{kr_{n}}{n}, \frac{Y_{i+(k-1)r_{n}}}{a_{n}}
 \bigg) \bigg\} \bigg]
 - \prod_{k=1}^{k_{n}} \mathrm{E} \bigg[ \exp \bigg\{ - \sum_{i=1}^{r_{n}} f \bigg(\frac{kr_{n}}{n}, \frac{Y_{i+(k-1)r_{n}}}{a_{n}} \bigg)
 \bigg\} \bigg] \bigg| \\[1em]
   & =: & V_{1}(n) + V_{2}(n) + V_{3}(n) + V_{4}(n).
\end{eqnarray}
 The function $f$ is nonnegative, bounded (by $M>0$ let us suppose)
 and its support is bounded away from origin, which implies that $f(t,x)=0$ for all $t \in
 [0,1]$ and $ |x| \leq \delta $ for some $\delta >0$. Denote by
 $j_{n}=n-k_{n}r_{n}$. Then using the inequality
 $1-e^{-x} \leq x$ for any $x \geq 0$, relation $(\ref{e:scaleparam1})$ and stationarity of the sequence $(Z_{i})_{i}$ we obtain
 \begin{eqnarray}\label{e:I1}
   \nonumber V_{1}(n) & \leq & \mathrm{E} \bigg[ \exp \bigg\{ - \sum_{i=1}^{k_{n}r_{n}}
                                    f \bigg(\frac{i}{n}, \frac{Y_{i}}{a_{n}} \bigg) \bigg\} \cdot \bigg|
                                    1 - \exp \bigg\{ - \sum_{i=k_{n}r_{n}+1}^{n}
                                    f \bigg(\frac{i}{n}, \frac{Y_{i}}{a_{n}} \bigg) \bigg\} \bigg| \bigg]
                                    \\[0.5em]
   \nonumber & \leq & \mathrm{E} \bigg[ \sum_{i=k_{n}r_{n}+1}^{n} f \bigg(\frac{i}{n}, \frac{Y_{i}}{a_{n}}
                   \bigg) \bigg] =   \sum_{i=k_{n}r_{n}+1}^{n} \mathrm{E} \bigg[ f \bigg(\frac{i}{n}, \frac{Y_{i}}{a_{n}} \bigg)
                   1_{\big\{ \frac{|Y_{i}|}{a_{n}} > \delta \big\} }
                   \bigg] \\[0.5em]
   \nonumber & \leq & M  \sum_{i=k_{n}r_{n}+1}^{n} \mathrm{P} (|Y_{i}|>\delta a_{n}) \leq M  \sum_{i=k_{n}r_{n}+1}^{n} \mathrm{P} (\beta_{2} |Z_{i}|>\delta a_{n})\\[0.6em]
    & = & M j_{n} \mathrm{P} (|Z_{1}|> \beta_{2}^{-1} \delta a_{n}).
 \end{eqnarray}
In a similar manner we obtain
\begin{eqnarray}\label{e:I2}
   \nonumber V_{2}(n) &\leq&   \mathrm{E} \bigg[ \exp \bigg\{ - \sum_{k=1}^{k_{n}} \sum_{i=(k-1)r_{n}+1}^{kr_{n}-l_{n}}
                                    f \bigg(\frac{i}{n}, \frac{Y_{i}}{a_{n}} \bigg) \bigg\} \cdot \bigg|
                                    1 - \exp \bigg\{ - \sum_{k=1}^{k_{n}}\sum_{i=kr_{n}-l_{n}+1}^{kr_{n}}
                                    f \bigg(\frac{i}{n}, \frac{Y_{i}}{a_{n}} \bigg) \bigg\} \bigg| \bigg]
                                    \\[0.5em]
     &\leq& M \sum_{k=1}^{k_{n}}\sum_{i=kr_{n}-l_{n}+1}^{kr_{n}} \mathrm{P} (\beta_{2} |Z_{i}|>\delta a_{n}) =  M k_{n}l_{n} \mathrm{P} (|Z_{1}|> \beta_{2}^{-1}\delta a_{n}).
\end{eqnarray}
Further we have
$$ V_{3}(n) \leq V_{5}(n) + V_{6}(n) + V_{7}(n),$$
where
\begin{eqnarray*}
V_{5}(n) & = & \bigg| \mathrm{E} \bigg[ \exp \bigg\{ - \sum_{k=1}^{k_{n}} \sum_{i=(k-1)r_{n}+1}^{kr_{n}-l_{n}}
      f \bigg( \frac{i}{n}, \frac{Y_{i}}{a_{n}} \bigg) \bigg\} \bigg] \\[0.3em]
   & &  \hspace*{-2.2em} - \ \mathrm{E} \bigg[ \exp \bigg\{ - \sum_{i=1}^{r_{n}-l_{n}} f \bigg(\frac{i}{n}, \frac{Y_{i}}{a_{n}} \bigg) \bigg\} \bigg]
     \mathrm{E} \bigg[ \exp \bigg\{ - \sum_{k=2}^{k_{n}} \sum_{i=(k-1)r_{n}+1}^{kr_{n}-l_{n}}
    f \bigg(\frac{i}{n}, \frac{Y_{i}}{a_{n}} \bigg) \bigg\} \bigg] \bigg|,
\end{eqnarray*}
\begin{eqnarray*}
V_{6}(n) & = & \bigg| \mathrm{E} \bigg[ \exp \bigg\{ - \sum_{i=1}^{r_{n}-l_{n}} f \bigg(\frac{i}{n}, \frac{Y_{i}}{a_{n}} \bigg) \bigg\} \bigg]
     \mathrm{E} \bigg[ \exp \bigg\{ - \sum_{k=2}^{k_{n}} \sum_{i=(k-1)r_{n}+1}^{kr_{n}-l_{n}}
     f \bigg(\frac{i}{n}, \frac{Y_{i}}{a_{n}} \bigg) \bigg\} \bigg]\\[0.3em]
   & & \hspace*{-2.2em} - \ \mathrm{E} \bigg[ \exp \bigg\{ - \sum_{i=1}^{r_{n}-l_{n}} f \bigg(\frac{1 \cdot r_{n}}{n}, \frac{Y_{i}}{a_{n}} \bigg) \bigg\} \bigg]
     \mathrm{E} \bigg[ \exp \bigg\{ - \sum_{k=2}^{k_{n}} \sum_{i=(k-1)r_{n}+1}^{kr_{n}-l_{n}}
     f \bigg(\frac{i}{n}, \frac{Y_{i}}{a_{n}} \bigg) \bigg\} \bigg]\bigg|,
\end{eqnarray*}
and
\begin{eqnarray*}
V_{7}(n) & = & \bigg| \mathrm{E} \bigg[ \exp \bigg\{ - \sum_{i=1}^{r_{n}-l_{n}} f \bigg(\frac{1 \cdot r_{n}}{n}, \frac{Y_{i}}{a_{n}} \bigg) \bigg\} 
     \mathrm{E} \bigg[ \exp \bigg\{ - \sum_{k=2}^{k_{n}} \sum_{i=(k-1)r_{n}+1}^{kr_{n}-l_{n}}
     f \bigg(\frac{i}{n}, \frac{Y_{i}}{a_{n}} \bigg) \bigg\} \bigg]\\[0.3em]
  & & \hspace*{-2.2em} - \ \prod_{k=1}^{k_{n}} \mathrm{E} \bigg[ \exp \bigg\{ - \sum_{i=1}^{r_{n}-l_{n}}
      f \bigg(\frac{kr_{n}}{n}, \frac{Y_{i+(k-1)r_{n}}}{a_{n}} \bigg) \bigg\} \bigg]\bigg|.
\end{eqnarray*}
The inequality
$ | \mathrm{E} (gh) - \mathrm{E} g\,\mathrm{E} h| \leq 4C_{1}C_{2} \alpha_{m}$,
for a $\mathcal{F}_{-\infty}^{j}$ measurable function $g$ and a
$\mathcal{F}_{j+m}^{\infty}$ measurable function $h$ such that $|g|
\leq C_{1}$ and $|h| \leq C_{2}$ (Lin and Lu~\cite{LiLu97}, Lemma 1.2.1), gives
\begin{equation}\label{e:I5}
    V_{5}(n) \leq 4 \alpha_{l_{n}+1}.
\end{equation}
For any $t>0$ there exists a constant $C(t)>0$ such that the
following inequality holds:
$$ |1-e^{-x}| \leq C(t)|x| \qquad \textrm{for all} \ |x| \leq t.$$
Further, for arbitrary real numbers $z_{1}, \ldots, z_{n}$ and $w_{1}, \ldots, w_{n}$ it holds that
\begin{equation}\label{eq:durrett}
\bigg| \prod_{k=1}^{n}z_{k} - \prod_{k=1}^{n}w_{k} \bigg| \leq A^{n-1} \sum_{k=1}^{n}|z_{k}-w_{k}|
\end{equation}
where $A= \max \{|z_{1}|,\ldots, |z_{n}|, |w_{1}|,\ldots, |w_{n}| \}$.
These last two inequalities imply
\begin{eqnarray*}
  V_{6}(n) & \leq & \mathrm{E} \bigg| \exp \bigg\{ - \sum_{i=1}^{r_{n}-l_{n}}
       f \bigg(\frac{i}{n}, \frac{Y_{i}}{a_{n}} \bigg) \bigg\} - \exp \bigg\{ - \sum_{i=1}^{r_{n}-l_{n}}
       f \bigg(\frac{r_{n}}{n}, \frac{Y_{i}}{a_{n}} \bigg) \bigg\} \bigg|
       \\[0.3em]
   & \leq &  \sum_{i=1}^{r_{n}-l_{n}} \mathrm{E} \bigg| \exp \bigg\{ -f \bigg(\frac{i}{n}, \frac{Y_{i}}{a_{n}} \bigg) \bigg\}
       - \exp \bigg\{ -f \bigg(\frac{r_{n}}{n}, \frac{Y_{i}}{a_{n}} \bigg) \bigg\}
       \bigg|\\[0.3em]
   & \leq & \sum_{i=1}^{r_{n}-l_{n}} \mathrm{E} \bigg| 1
       - \exp \bigg\{ f \bigg(\frac{i}{n}, \frac{Y_{i}}{a_{n}} \bigg) -
       f \bigg(\frac{r_{n}}{n}, \frac{Y_{i}}{a_{n}} \bigg) \bigg\}
       \bigg|\\[0.3em]
   & \leq & C(2M) \sum_{i=1}^{r_{n}-l_{n}} \mathrm{E} \bigg| f \bigg(\frac{i}{n}, \frac{Y_{i}}{a_{n}} \bigg) -
       f \bigg(\frac{r_{n}}{n}, \frac{Y_{i}}{a_{n}} \bigg)
       \bigg|.
\end{eqnarray*}
Therefore
\begin{eqnarray*}
  V_{6}(n) & \leq & C(2M) \sum_{i=1}^{r_{n}-l_{n}} \mathrm{E} \bigg[ \bigg| f \bigg(\frac{i}{n}, \frac{Y_{i}}{a_{n}} \bigg) -
       f \bigg(\frac{r_{n}}{n}, \frac{Y_{i}}{a_{n}} \bigg)
       \bigg| 1_{ \big\{ \frac{|Y_{i}|}{a_{n}} > \delta \big\} } \bigg].
\end{eqnarray*}
Since a continuous function on a compact set is uniformly
continuous, it follows that for any $\epsilon >0$ there exists
$\gamma >0$ such that for $(s,x), (s',x') \in [0,1] \times \{y \in
\mathbb{E} : |y| > \delta \}$, if $d((s,x),(s',x')) < \gamma$ then $|f(s,x) - f(s',x')|<
\epsilon$, where $d$ is the
metric on the direct product of metric spaces $[0,1]$ and
$\mathbb{E}$, i.e.
$ d((s,x),(s',x')) = \max \{ |s-s'|,
\rho(x,x') \}$ with $\rho$ defined in (\ref{e:metricrho}). Since $r_{n}/n \to 0$ as $n \to \infty$,
for large $n$ we have
$$ d \bigg( \bigg(\frac{i}{n}, \frac{Y_{i}}{a_{n}}\bigg),
  \bigg(\frac{r_{n}}{n}, \frac{Y_{i}}{a_{n}}\bigg) \bigg) = \frac{|i
   -r_{n}|}{n} \leq \frac{r_{n}}{n} < \gamma,$$
for any $i=1, \ldots , r_{n}-l_{n}$. Therefore, for large $n$,
$$ \bigg| f \bigg(\frac{i}{n}, \frac{Y_{i}}{a_{n}} \bigg) -
       f \bigg(\frac{r_{n}}{n}, \frac{Y_{i}}{a_{n}} \bigg)
       \bigg| < \epsilon,$$ and this implies
\begin{eqnarray}\label{e:I6}
  \nonumber V_{6}(n) &\leq&  \epsilon\,C(2M) \sum_{i=1}^{r_{n}-l_{n}} \mathrm{P} (|Y_{i}|>\delta a_{n})
  \leq  \epsilon\,C(2M) \sum_{i=1}^{r_{n}-l_{n}} \mathrm{P} (\beta_{2}|Z_{i}|>\delta a_{n})\\[0.5em]
   &\leq& \epsilon\,C(2M)(r_{n}-l_{n}) \mathrm{P} (|Z_{1}| > \beta_{2}^{-1}\delta a_{n})
\end{eqnarray}
for large $n$.
Taking into account relations (\ref{e:I5}) and (\ref{e:I6}), it
follows that, for large $n$,
$$ V_{3}(n) \leq 4 \alpha_{l_{n}+1} + \epsilon\,C(2M) r_{n}
\mathrm{P}(|Z_{1}| > \beta_{2}^{-1}\delta a_{n}) + V_{7}(n),$$
 and since it is easy to obtain
\begin{equation*}
 V_{7}(n) \leq \bigg|
     \mathrm{E} \bigg[ \exp \bigg\{ - \sum_{k=2}^{k_{n}} \sum_{i=(k-1)r_{n}+1}^{kr_{n}-l_{n}}
     f \bigg(\frac{i}{n}, \frac{Y_{i}}{a_{n}} \bigg) \bigg\} \bigg] - \prod_{k=2}^{k_{n}} \mathrm{E}
      \bigg[ \exp \bigg\{ - \sum_{i=1}^{r_{n}-l_{n}}
      f \bigg(\frac{kr_{n}}{n}, \frac{Y_{i+(k-1)r_{n}}}{a_{n}} \bigg) \bigg\} \bigg]
      \bigg|,
\end{equation*}
we recursively obtain (we repeat the same procedure for $V_{7}(n)$
as we did for $V_{3}(n)$ and so on)
\begin{equation}\label{e:I3}
    V_{3}(n) \leq 4 k_{n} \alpha_{l_{n}+1} + \epsilon\,C(2M)
    k_{n}r_{n} \mathrm{P} (|Z_{1}|>\beta_{2}^{-1}\delta a_{n}).
\end{equation}
Finally (\ref{eq:durrett}) and the same procedure that we used before for $V_{1}(n)$ give
\begin{eqnarray}\label{e:I4}
\nonumber V_{4}(n) &\leq & \sum_{k=1}^{k_{n}} \mathrm{E} \bigg| \exp \bigg\{ - \sum_{i=1}^{r_{n}-l_{n}} f \bigg(\frac{kr_{n}}{n}, \frac{Y_{i+(k-1)r_{n}}}{a_{n}}
 \bigg) \bigg\}
 -  \exp \bigg\{ - \sum_{i=1}^{r_{n}} f \bigg(\frac{kr_{n}}{n}, \frac{Y_{i+(k-1)r_{n}}}{a_{n}} \bigg)
 \bigg\} \bigg|\\[0.5em]
 \nonumber &\leq& \sum_{k=1}^{k_{n}} \mathrm{E} \bigg| 1
 -  \exp \bigg\{ - \sum_{i=r_{n}-l_{n}+1}^{r_{n}} f \bigg(\frac{kr_{n}}{n}, \frac{Y_{i+(k-1)r_{n}}}{a_{n}} \bigg)
 \bigg\} \bigg|\\[0.5em]
 \nonumber &\leq& \sum_{k=1}^{k_{n}} \mathrm{E} \bigg[ \sum_{i=r_{n}-l_{n}+1}^{r_{n}} f \bigg(\frac{kr_{n}}{n}, \frac{Y_{i+(k-1)r_{n}}}{a_{n}} \bigg)
  \bigg]\\[0.5em]
  &\leq&  M k_{n}l_{n} \mathrm{P} (|Z_{1}|> \beta_{2}^{-1}\delta a_{n})
\end{eqnarray}
Thus from (\ref{e:I}), (\ref{e:I1}), (\ref{e:I2}),
(\ref{e:I3}) and (\ref{e:I4}) it follows that for large $n$,
\begin{equation*}
  V(n) \leq \bigg( M \frac{j_{n}}{n} + 2M \frac{k_{n}l_{n}}{n} +
         \epsilon\,C(2M) \frac{k_{n}r_{n}}{n} \bigg) \cdot n \mathrm{P}(|Z_{1}|> \beta_{2}^{-1}\delta a_{n})
    + \ 4 k_{n} \alpha_{l_{n}+1}.
\end{equation*}
 Since $Z_{1}$ is regularly varying, by (\ref{e:seqan1}) it holds
 that
 $ n\mathrm{P}(|Z_{1}|> \beta_{2}^{-1}\delta a_{n})\rightarrow \beta_{2}^{\alpha}\delta^{-\alpha}$
 as $n \rightarrow \infty$.
 This, together with relation
  (\ref{e:strmix1}) and the fact that $j_{n}/n \rightarrow 0$ and $k_{n}r_{n}/n \rightarrow 1$
 as
 $n \rightarrow \infty$, implies
$$ \limsup_{n \rightarrow \infty} V(n) \leq \epsilon\,C(2M)
    \beta_{2}^{\alpha}\delta^{-\alpha}.$$
But since this holds for all $\epsilon >0$, we get $ \lim_{n
\rightarrow \infty} V(n) = 0$, and thus condition $\mathcal{A}'$
 holds.
\end{proof}

The fact that relation (\ref{e:DaRe15}) in the proof of Theorem~\ref{t:mainppc} holds under (\ref{e:DaRe16}) is justified by the following result. Let $\mathbb{S}$ be a locally compact topological space with countable base. For $\eta \in M_{+}(\mathbb{S})$ and $f \in C_{K}^{+}(\mathbb{S})$ put
$$\eta (f) = \int_{\mathbb{S}} f(x)\,\eta(\rmd x).$$

\begin{lemma}\label{l:convprob}
 Let $(\xi_{n})_{n}$ and $(\eta_{n})_{n}$ be two sequences of random measures, i.e.~ sequences of random elements of $M_{+}(\mathbb{S})$ such that for all $f \in C_{K}^{+}(\mathbb{S})$
 $$  \xi_{n}(f)-\eta_{n}(f) \Pto 0 \qquad \textrm{as} \ n \to \infty.$$
Then, as $n \to \infty$,
$$ \d_{v}( \xi_{n}, \eta_{n} ) \Pto 0,$$
where $d_{v}$ is the vague metric on $M_{+}(\mathbb{S})$.
\end{lemma}
\begin{proof}
Let $\epsilon >0$ be arbitrary, and take $k_{0} \in \mathbb{N}$ such that
$\sum_{k=k_{0}+1}^{\infty}2^{-k} < \epsilon/2$. Then for every $f \in C_{K}^{+}(\mathbb{S})$
\begin{equation*}\label{l:vagmetricP1}
\sum_{k=k_{0}+1}^{\infty} 2^{-k} \big(1-e^{-|\xi_{n}(f)-\eta_{n}(f)|} \big) \leq \sum_{k=k_{0}+1}^{\infty} 2^{-k} < \frac{\epsilon}{2}.
\end{equation*}
This, together with the definition of the metric $d_{v}$ in $(\ref{e:vaguemetric})$ and the inequality $1-e^{-x} \leq x$ for $x \geq 0$, implies
\begin{eqnarray*}
\Pr[ \d_{v}( \xi_{n}, \eta_{n} ) > \epsilon]  &\leq& \Pr \bigg( \sum_{k=1}^{k_{0}}2^{-k} \big(1-e^{-|\xi_{n}(f_{k})-\eta_{n}(f_{k})|} \big) > \frac{\epsilon}{2} \bigg)\\[0.2em]
&\leq& \Pr \bigg( \sum_{k=1}^{k_{0}}2^{-k} |\xi_{n}(f_{k})-\eta_{n}(f_{k})| > \frac{\epsilon}{2} \bigg),
\end{eqnarray*}
where $(f_{k})_{k}$ is a certain sequence of functions in $C_{K}^{+}(\mathbb{S})$. Hence
$$ \Pr[ \d_{v}( \xi_{n}, \eta_{n} ) > \epsilon] \leq \sum_{k=1}^{k_{0}} \Pr \bigg( |\xi_{n}(f_{k})-\eta_{n}(f_{k})| > \frac{ \epsilon}{2k_{0}} \bigg),$$
and since by assumption $\xi_{n}(f_{k})-\eta_{n}(f_{k}) \Pto 0$ as $n \to \infty$ for every $k=1,\ldots,k_{0}$, we conclude that $ \d_{v}( \xi_{n}, \eta_{n} ) \Pto 0$.
\end{proof}

\begin{lemma}\label{l:ItoPRM}
 Let $N = \sum_{k} \delta_{(t_{k},j_{k})}$ be a Poisson process on $[0,1] \times \mathbb{E}$ with intensity measure $Leb \times \nu$, where $\mathbb{E} = [-\infty, \infty] \setminus \{0\}$ and
 $$\nu(\rmd x)= \lambda (p\,1_{(0,\infty)}(x) + q\,1_{(-\infty,0)}(x))\,\alpha |x|^{-\alpha-1} \rmd x,$$
  with $\lambda >0$, $\alpha > 0$, $p \in [0,1]$ and $q=1-p$. Let $b$ be a positive real number. Then
 $$ N^{+} := \sum_{k} \delta_{(t_{k}, b j_{k}1_{\{j_{k} >0\}})},$$
 is a Poisson process on $[0,1] \times \mathbb{E}$ with intensity measure $Leb \times \nu^{+}$, where
 $$ \nu^{+}(\rmd x) = \lambda p b^{\alpha}\alpha x^{-\alpha-1}1_{(0,\infty)}(x) \rmd x,$$
 and
  $$ N^{-} := \sum_{k} \delta_{(t_{k}, -b j_{k}1_{\{j_{k} <0\}})},$$
 is a Poisson process on $[0,1] \times \mathbb{E}$ with intensity measure $Leb \times \nu^{-}$, where
 $$ \nu^{-}(\rmd x) = \lambda q b^{\alpha}\alpha x^{-\alpha-1}1_{(0,\infty)}(x) \rmd x,$$
\end{lemma}
\begin{proof}
Define $T_{1} \colon \mathbb{E} \to \mathbb{E}$ by $T_{1}(x)=b x 1_{\{ x >0 \}}$. By Proposition 5.2 in Resnick~\cite{Resnick07}
$$ \sum_{k} \delta_{(t_{k}, T_{1}(j_{k}))} = \sum_{k} \delta_{(t_{k}, b j_{k} 1 _{\{ j_{k}> 0\}})}$$
is a Poisson process with intensity measure $Leb \times (\nu \circ T_{1}^{-1})$.
For $x>0$ we have
\begin{eqnarray*}
(\nu \circ T_{1}^{-1})((x,\infty]) &=& \nu (\{ t \in \mathbb{E} : T_{1}(t) > x\}) = \nu (\{ t \in \mathbb{E} :  b t 1_{\{ t>0\}} >x \})= \nu \Big( \Big(\frac{x}{b}, \infty \Big] \Big)\\[0.6em]
&=& \lambda p \alpha \int_{x/b}^{\infty}x^{-\alpha-1}\,\rmd x = \lambda p b^{\alpha} x^{-\alpha},
\end{eqnarray*}
and
$$ (\nu \circ T_{1}^{-1})([-\infty,-x)) = \nu (\{ t \in \mathbb{E} :  t 1_{\{ t>0\}}b < -x \}) = \nu (\emptyset) =0.$$
Therefore
\begin{equation*}
(\nu \circ T_{1}^{-1})(\rmd x) =  \lambda p b^{\alpha} \alpha x^{-\alpha-1} 1_{(0,\infty)}(x) \rmd x.
\end{equation*}
The statement about $N_{-}$ follows similarly by taking $T_{2}(x)= -b x 1_{\{ x <0 \}}$ instead of $T_{1}$.
\end{proof}

\section*{Acknowledgements}
 This work has been supported by the European Union-NextGenerationEU-StatDinMatAn-uniri-iz-25-108 and MMUI-uniri-iz-25-74.


\end{document}